\documentclass[12pt,leqno]{article}
\usepackage{amsfonts}
\usepackage{amsmath, dsfont, amsthm, amsfonts, amssymb, color}
\usepackage{mathrsfs}
\usepackage{color}
\usepackage{stmaryrd}
\newtheorem{thm}{Theorem}[section]

\newtheorem{lem}[thm]{Lemma}

\newtheorem{rem}[thm]{Remark}
\theoremstyle{definition}
\newtheorem{defn}{Definition}[section]
\newcommand{\scr}[1]{\mathscr #1}
\definecolor{wco}{rgb}{0.5,0.2,0.3}

\numberwithin{equation}{section} \theoremstyle{remark}

\newcommand{\ua}{\uparrow}

\title{{\bf Coupling Methods and Applications on Path Dependent McKean-Vlasov SDEs}\footnote{Supported in
 part by  National Key R\&D Program of China (No. 2022YFA1006000) and NNSFC (12271398).\ \ }
}%Xiaochen Ma is the corresponding author.
\author{
	{\bf   Xing Huang $^{a)}$, Xiaochen Ma $^{b)}$   }\\
	\footnotesize{ a)Center for Applied Mathematics, Tianjin
		University, Tianjin 300072, China}\\
\footnotesize{ b)Department of Mathematics, City University of Hong Kong, Tat Chee Avenue, Hong Kong, China}\\
	\footnotesize{  xinghuang@tju.edu.cn, }
	\footnotesize{   xiaocma2-c@my.cityu.edu.hk}}
\begin{document}
\allowdisplaybreaks
\def\R{\mathbb R}  \def\ff{\frac} \def\ss{\sqrt} \def\B{\mathbf
B} \def\W{\mathbb W}
\def\N{\mathbb N} \def\kk{\kappa} \def\m{{\bf m}}
\def\ee{\varepsilon}\def\ddd{D^*}
\def\dd{\delta} \def\DD{\Delta} \def\vv{\varepsilon} \def\rr{\rho}
\def\<{\langle} \def\>{\rangle} \def\GG{\Gamma} \def\gg{\gamma}
  \def\nn{\nabla} \def\pp{\partial} \def\E{\mathbb E}
\def\d{\text{\rm{d}}} \def\bb{\beta} \def\aa{\alpha} \def\D{\scr D}
  \def\si{\sigma} \def\ess{\text{\rm{ess}}}
\def\beg{\begin} \def\beq{\begin{equation}}  \def\F{\scr F}
\def\Ric{\text{\rm{Ric}}} \def\Hess{\text{\rm{Hess}}}
\def\e{\text{\rm{e}}} \def\ua{\underline a} \def\OO{\Omega}  \def\oo{\omega}
 \def\tt{\tilde} \def\Ric{\text{\rm{Ric}}}
\def\cut{\text{\rm{cut}}} \def\P{\mathbb P} \def\ifn{I_n(f^{\bigotimes n})}
\def\C{\scr C}      \def\aaa{\mathbf{r}}     \def\r{r}
\def\gap{\text{\rm{gap}}} \def\prr{\pi_{{\bf m},\varrho}}  \def\r{\mathbf r}
\def\Z{\mathbb Z} \def\vrr{\varrho}
\def\L{\scr L}\def\Tt{\tt} \def\TT{\tt}\def\II{\mathbb I}
\def\i{{\rm in}}\def\Sect{{\rm Sect}}  \def\H{\mathbb H}
\def\M{\scr M}\def\Q{\mathbb Q} \def\texto{\text{o}}
\def\Rank{{\rm Rank}} \def\B{\scr B} \def\i{{\rm i}} \def\HR{\hat{\R}^d}
\def\to{\rightarrow}\def\l{\ell}\def\iint{\int}
\def\EE{\scr E}\def\Cut{{\rm Cut}}
\def\A{\scr A} \def\Lip{{\rm Lip}}
\def\BB{\scr B}\def\Ent{{\rm Ent}}\def\L{\scr L}
\def\R{\mathbb R}  \def\ff{\frac} \def\ss{\sqrt} \def\B{\mathbf
B}
\def\N{\mathbb N} \def\kk{\kappa} \def\m{{\bf m}}
\def\dd{\delta} \def\DD{\Delta} \def\vv{\varepsilon} \def\rr{\rho}
\def\<{\langle} \def\>{\rangle} \def\GG{\Gamma} \def\gg{\gamma}
  \def\nn{\nabla} \def\pp{\partial} \def\E{\mathbb E}
\def\d{\text{\rm{d}}} \def\bb{\beta} \def\aa{\alpha} \def\D{\scr D}
  \def\si{\sigma} \def\ess{\text{\rm{ess}}}
\def\beg{\begin} \def\beq{\begin{equation}}  \def\F{\scr F}
\def\Ric{\text{\rm{Ric}}} \def\Hess{\text{\rm{Hess}}}
\def\e{\text{\rm{e}}} \def\ua{\underline a} \def\OO{\Omega}  \def\oo{\omega}
 \def\tt{\tilde} \def\Ric{\text{\rm{Ric}}}
\def\cut{\text{\rm{cut}}} \def\P{\mathbb P} \def\ifn{I_n(f^{\bigotimes n})}
\def\C{\scr C}      \def\aaa{\mathbf{r}}     \def\r{r}
\def\gap{\text{\rm{gap}}} \def\prr{\pi_{{\bf m},\varrho}}  \def\r{\mathbf r}
\def\Z{\mathbb Z} \def\vrr{\varrho}
\def\L{\scr L}\def\Tt{\tt} \def\TT{\tt}\def\II{\mathbb I}
\def\i{{\rm in}}\def\Sect{{\rm Sect}}  \def\H{\mathbb H}
\def\M{\scr M}\def\Q{\mathbb Q} \def\texto{\text{o}} \def\LL{\Lambda}
\def\Rank{{\rm Rank}} \def\B{\scr B} \def\i{{\rm i}} \def\HR{\hat{\R}^d}
\def\to{\rightarrow}\def\l{\ell}
\def\8{\infty}\def\I{1}\def\U{\scr U} \def\n{{\mathbf n}}
\maketitle

\begin{abstract} By using coupling by change of conditional probability measure, the log-Harnack inequality for path dependent McKean-Vlasov SDEs with distribution dependent diffusion coefficients is established, which together with the exponential contractivity in $L^2$-Wasserstein distance yields the exponential contractivity in entropy-cost under the uniformly dissipative condition. When the coefficients are only partially dissipative, the exponential contractivity in $L^1$-Wasserstein distance is also derived in the aid of asymptotic reflection coupling, which is new even in the distribution independent case. In addition, the uniform in time propagation of chaos in $L^1$- Wasserstein distance is also obtained for path dependent mean field interacting particle system.
	%The well-posedness, probability distance estimates and the log-Harnack inequality are derived for path-dependent McKean-Vlasov SDEs with distribution dependent noise, where the coefficients are  ($W_2^\mathscr{C}$ + $W_\alpha$) - Lipschitz continuous in the distribution variable. This improves existing results in which the coefficients are Lipschitz continuous in $L^2$-Wasserstein distance.\\

\end{abstract} \noindent
 \vskip 2cm

\section{Introduction}
Let $d\geq 1$ and fix a constant $r_0>0$. Let $\C=C([-r_0,0];\R^{d})$, the set of $\R^{d}$-valued continuous functions on $[-r_0,0]$. For any $f\in C([-r_0,\infty);\R^{d})$, $t\geq 0$, define $f_t\in\C$ as $f_t(s)=f(t+s)$, $s\in[-r_0,0]$, which is called the segment process. When $\C$ is equipped with the uniform norm $\|\xi\|_{\infty}:= \sup_{s\in[-r_0,0]}|\xi(s)|$, we denote it as $\C_\infty$.
Let $\mathscr{B}(\C_\infty)$ be the Borelian $\sigma$-field on $\C_\infty$ and $\mathscr{P}(\C_\infty)$ be the set of all probability measures on $(\C_\infty,\mathscr{B}(\C_\infty))$ equipped with the weak topology. Let $\beta \neq0$ be a constant and consider the following path dependent McKean-Vlasov SDEs:
\beq\label{EQ}\d X(t)=B(X_t,\L_{X_t})\d t+\beta\d W^1(t)+\sigma(X_t,\L_{X_t})\d W^2(t),\end{equation}
 where $\{W^1(t)\}_{t\geq 0}$ and $\{W^2(t)\}_{t\geq0}$ are two independent $d$-dimensional and $d_2$-dimensional Brownian motions on a complete filtration probability space $(\Omega,\{\mathscr{F}_t\}_{t\geq 0},\mathscr{F},\P)$, $\L_{X_t}$ is the law of $X_t$, and
 $$B:\C_\infty\times\mathscr{P}(\C_\infty)\to\mathbb{R}^d,\ \  \sigma:\C_\infty\times\mathscr{P}(\C_\infty)\to \mathbb{R}^{d}\otimes\R^{d_2}$$ are measurable. Note that if we set $\Sigma=\sqrt{\beta I_{d\times d}+\sigma\sigma^\ast}$ and consider
\beq\label{EQc}\d \tilde{X}(t)=B(\tilde{X}_t,\mathscr L_{\tilde{X}_t})\d t+\Sigma(\tilde{X}_t,\mathscr L_{\tilde{X}_t})\d W(t),\end{equation}
for some $d$-dimensional Brownian motion $W(t)$ . Then the probability distribution of \eqref{EQ} coincides with that of \eqref{EQc} if the weak uniqueness holds. Since it is more convenient for us to construct couplings basing on \eqref{EQ}, we will study \eqref{EQ} instead of \eqref{EQc} throughout the paper.
For any $k\geq 1$, let
$$\mathscr{P}_k(\C_\infty):=\big\{\mu\in\mathscr{P}(\C_\infty):\mu(\|\cdot\|^k_\infty)<\infty\big\}. $$
 $\mathscr{P}_k(\C_\infty)$ is a Polish space under the $L^k$-Wasserstein distance
$$\W^{\C_\infty}_k(\mu,\nu):= \inf_{\pi\in\textbf{C}(\mu,\nu)}\left(\int_{\mathscr{C}\times \C}\|\xi-\tilde{\xi}\|^k_{\infty}\pi(\d \xi,\d \tilde{\xi})\right)^{\frac{1} {k}},\ \ \mu,\nu\in\scr P_k(\C_\infty),$$where $\textbf{C}(\mu,\nu)$ is the set of all couplings for $\mu$ and $\nu$. Similarly, we can define $\scr P(\R^d)$ and $\scr P_k(\R^d)$.

\begin{defn}\label{DEF} (i) A continuous process $(X(t))_{t\geq
-r_0}$ on $\mathbb{R}^d$ is called a strong solution of \eqref{EQ} with some $\F_0$-measurable and $\C_\infty$-valued initial value $X_0$, if for any $t\geq 0$, $X(t)$ is $\F_t$-measurable,
and $\P$-a.s.
\begin{align}\label{integ}\nonumber X(t)&=X(t\wedge 0)+\int_0^{t\vee 0}B(X_s,\L_{X_s})\d s+\beta W^1(t\vee 0)\\
&+\int_0^{t\vee 0}\si(X_s,\L_{X_s})\d W^2(s), ~~~~t\ge0.
\end{align}
(ii) For any $\mu_0\in\scr P(\C_\infty)$, $(\tilde{\Omega},\{\tilde{\F}_t\}_{t\geq 0},\tilde{\P},\tilde{W}^1,\tilde{W}^2, \{\tilde{X}(t)\}_{t\geq-r_0})$ is called a weak solution starting from $\mu_0$, if $\L_{\tilde{X}_0}=\mu_0$
and $\tilde{\P}$-a.s. \eqref{integ} holds for  $(\tilde{W}^1,\tilde{W}^2, \{\tilde{X}(t)\}_{t\geq-r_0})$ replacing $(W^1,W^2, \{X(t)\}_{t\geq-r_0})$.

(iii) For $k\geq 1$, the SDE \eqref{EQ} is called well-posed for distributions in $\scr P_k(\C_\infty)$, if for any $\F_0$-measurable initial value
	$X_0$ with $\L_{X_0}\in \scr P_k(\C_\infty)$ (respectively any initial distribution $\gamma\in\scr P_k(\C_\infty)$), it has a unique strong solution (respectively weak solution) such that $\L_{X(\cdot)}\in C([0,\infty);\scr P_k(\R^d))$.
\end{defn}

 When \eqref{EQ} is well-posed for distributions in $\mathscr{P}_k(\C_\infty)$ for some $k\geq 1$, let
$P^{\ast}_t\mu:= \L_{X_t^\mu}, t\geq 0$ for $\L_{X_0^\mu}=\mu\in\mathscr{P}_k(\C_\infty).$
Recall that for two probability measures $\mu,\nu\in\scr P(E)$, the set of all probability measures on some measurable space $(E,\scr E)$, the relative entropy and total variation distance are defined as follows:
$$\Ent(\nu|\mu):= \beg{cases} \int_E (\log \ff{\d\nu}{\d\mu})\d\nu, \ &\text{if}\ \nu\ \text{ is\ absolutely\ continuous\ with\ respect\ to}\ \mu,\\
\infty,\ &\text{otherwise;}\end{cases}$$ and
$$\|\mu-\nu\|_{var} := \sup_{|f|\leq 1}|\mu(f)-\nu(f)|.$$
By Pinsker's inequality (see \cite{Pin}), it holds
\beq\label{ETX} \|\mu-\nu\|_{var}^2\le 2 \Ent(\nu|\mu),\ \ \mu,\nu\in \scr P(E).\end{equation}
The log-Harnack inequality for \eqref{EQ} is formulated as
\begin{align*}
(P_t^\ast\mu_0)(\log f)&\leq \log (P_t^\ast\nu_0)(f)+c(t)\W_2^{\C_\infty}(\mu_0,\nu_0)^2,\\
&\ \  f>0,f\in\scr B_b(\C_\infty),\mu_0,\nu_0\in\scr P_2(\C_\infty), t>r_0
\end{align*}
for some measurable function $c:(r_0,\infty)\to(0,\infty)$ with $\lim_{t\to r_0}c(t)=\infty$, which is equivalent to the entropy-cost inequality
\begin{align}\label{enc}\mathrm{Ent}(P_t^\ast\mu_0|P_t^\ast\nu_0)\leq c(t)\W_2^{\C_\infty}(\mu_0,\nu_0)^2, \ \ \mu_0,\nu_0\in\scr P_2(\C_\infty), t>r_0.
\end{align}
The stochastic functional differential equations (also named path dependent SDEs in the literature) was first proposed in \cite{53}. One can refer to the monograph \cite{29} for the research on stochastic functional differential equations. In \cite{28}, the authors investigated the existence and uniqueness of stochastic functional differential equations under one-sided Lipschitz condition. In \cite{SWY}, the dimension-free Harnack inequalities for stochastic (functional) differential equations with non-Lipschitzian coefficients are established by the method of coupling by change of measure.

The McKean-Vlasov SDE was proposed by McKean in \cite{M} and it has extensive applications in finance mathematics. In recent years, the path dependent McKean-Vlasov SDEs  have also attracted much attention.
\cite{L18} proves the propagation of chaos in total variation distance for path dependent mean field interacting particle system.
\cite{GS} establishes the large and moderate deviation principles for path-distribution dependent stochastic differential equations. \cite{RWu19} investigates the least squares estimator for path dependent McKean-Vlasov SDEs.
\cite{BPR} studies the large deviations for the empirical measure associated to path dependent interacting diffusions. \cite{Han} adopts the technique of relative entropy to derive the well-posedness of path dependent McKean-Vlasov SDEs driven by (fractional) Brownian motion. In \cite{HX}, the first author of the present paper proves the existence and uniqueness for path dependent McKean-Vlasov type SDEs with integrability conditions.
One can also refer to \cite{HRW} for the log-Harnack inequality, and \cite{BRW20} for the Bismut type formula.
So far, when the relative entropy is involved in, the diffusion coefficients are required to be distribution free so that the Girsanov transform is available.

Recently, in the path independent case, i.e. $r_0=0$, some progress has been made in the study of entropy-cost estimate for McKean-Vlasov SDEs with distribution dependent noise. In \cite{HW2209}, the first author and his collaborator adopt  a noise decomposition argument to obtain the log-Harnack inequality for McKean-Vlasov SDEs with diffusion depending only on measure variable. In \cite{23RW}, the authors develop the bi-coupling method to derive the log-Harnack inequality for McKean-Vlasov SDEs with diffusion depending on time, spatial and measure variables. All of these encourage us to investigate log-Harnack inequality for path dependent McKean-Vlasov SDEs with distribution dependent noise.

Before moving on, let us first show the difficulties in studying the log-Harnack inequality for path dependent McKean-Vlasov SDEs with distribution dependent noise. For simplicity, consider
\begin{align*}\d X(t)=\beta\d W^1(t)+\sigma(\L_{X_t})\d W^2(t).
\end{align*}
Then the solution $X^{\mu_0}(t)$ with initial distribution $\mu_0$ follows Gaussian distribution with covariance  matrix $\int_0^t(\beta I_{d\times d}+(\sigma\sigma^\ast)(P_s^\ast\mu_0))\d s$.
It is well-known that for invertible $d\times d$ matrices $\sigma_1,\sigma_2$, if $\sigma_1\neq\sigma_2$, $\{\sigma_1W^1(t)\}_{t\in[0,1]}$ is singular with $\{\sigma_2W^1(t)\}_{t\in[0,1]}$ on the path space $C([0,1],\R^d)$ equipped with the uniform norm. So, when the diffusion coefficients depend on the distribution, one may not expect to derive the entropy-cost estimate \eqref{enc} for the segment process $X_t$. Instead, we may establish the entropy-cost estimate (equivalently the log-Harnack inequality) for $X(t)$:
\begin{align*}\mathrm{Ent}((P_t^\ast\mu_0)\circ\pi_0^{-1}|(P_t^\ast\nu_0)\circ\pi_0^{-1})\leq \tilde{c}(t)\W_2^{\C_\infty}(\mu_0,\nu_0)^2, \ \ \mu_0,\nu_0\in\scr P_2(\C_\infty), t>0
\end{align*}
for $\pi_0:\mathscr{C}\rightarrow\R^d $ defined by $\pi_0(\xi)=\xi(0)$,\ \ $\xi\in\C$,  $$(\gamma\circ \pi_0^{-1})(A):=\gamma(\{\xi:\pi_0(\xi)\in A\}), \ \ A\in\mathscr B(\R^d),\gamma\in\scr P(\C_\infty)$$ and some measurable function $\tilde{c}:(0,\infty)\to(0,\infty)$ with $\lim_{t\to 0}\tilde{c}(t)=\infty$.
To this end, we will adopt the coupling by change of conditional probability measure given $W^2$, i.e. $\P(\cdot|W^2)$, see the proof of Theorem \ref{RL} below for more details.

On the other hand, there are lots of results on the ergodicity of path independent McKean-Vlasov SDEs, i.e. the case $r_0=0$. When the noise is distribution free, i.e. the diffusion coefficients do not depend on the distribution variable, the exponential ergodicity in relative entropy has been established in  \cite{RW} under uniformly dissipative condition. In \cite{W23}, the exponential ergodicity in quasi-Wasserstein distance is also derived even in fully non-dissipative case and the main technique is the (asymptotic) reflection coupling. One can also refer to \cite{EGZ,LMW,LWZ,M03,Schuh,Song,FYW1} and references therein for more results on exponential ergodicity in the sense of various Wasserstein distance. For the technique of asymptotic reflection coupling, one can refer to \cite{W15} for more details. In the path dependent case, \cite{HL} obtains the exponential ergodicity in relative entropy by the entropy-cost estimate as well as the transport inequality for McKean-Vlasov SDEs with distribution free diffusion coefficients.

In the case without delay, the exponential ergodicity in $L^1$-Wasserstein distance for distribution independent SDEs with partially dissipative assumption can be derived by reflection coupling. However, due to the path dependence, the classical reflection coupling is unavailable since the two constructed coupling processes may separate after the first meet. Instead, we will adopt the asymptotic reflection coupling. Observe that after using the It\^{o}-Tanaka formula for the distance between the two coupling processes, the quadratic variation process of the martingale has singularity, which may produce essential difficulty if we apply the BDG inequality. To overcome this difficulty, we will equip some integral type norm induced by an appropriate measure $\Gamma^{r_0}$ on $[-r_0,0]$ instead of the uniform norm on $C([-r_0,0],\R^d)$.

Finally, the quantitative propagation of chaos has become a hot topic in the field of stochastic analysis.   \cite{JW} introduces the entropy method to study the propagation of chaos in relative entropy, see also \cite{L21} for the sharp propagation of chaos in relative entropy by the BBGKY argument. For the uniform in time propagation of chaos, one can refer to \cite{GBM, LL,MRW} by the uniform in time log-Sobolev inequality for McKean-Vlasov SDEs, \cite{DEGZ,LWZ} by using asymptotic reflection coupling method, \cite{LMW} by the technique of asymptotic refined basic coupling in L\'{e}vy noise case.

In this paper, we first prove the log-Harnack inequality for path dependent McKean-Vlasov SDEs with distribution dependent noise, which together with the exponential contractivity in $L^2$-Wasserstein distance derived by synchronous coupling under the uniformly dissipative condition yields the exponential contractivity in entropy-cost. In the partially dissipative situation, we will apply the asymptotic reflection coupling to study the exponential contractivity in $L^1$-Wasserstein distance induced by some integral type norm with respect to an appropriate measure $\Gamma^{r_0}$ on $[-r_0,0]$. The asymptotic reflection coupling is also used to derive the uniform in time propagation of chaos in $L^1$-Wasserstein distance for path dependent mean field interacting system, where the drifts are assumed to be partially dissipative.

%Since $P^{*}_t\mu$ is uniquely determined by $$P_tf(\mu):= \E[f(X^{\mu}(t))]=\int_{\C}f\d P^{*}_t\mu,\ \ f\in \mathscr{B}_b(\R^d), $$
%where $\mathscr{B}_b(\C)$ is the space of bounded measurable functions on $\mathscr{C}$, we study the regularity of functionals
%$$\mathscr{P}_2(\mathscr{C})\ni\mu\mapsto P_tf(\mu),\ \ t\in (0,T], f\in\mathscr{B}_b(\R^d).$$

The main contributions of the paper include:
\begin{enumerate}
\item[(1)] When the diffusion coefficients only depend on the measure variable, the log-Harnack inequality for path dependent McKean-Vlasov SDEs is established.
\item[(2)]   The exponential contractivity in entropy-cost is derived under uniformly dissipative condition.
\item[(3)] Under partially dissipative assumption, the exponential contractivity in $L^1$-Wasserstein distance  is obtained for path dependent McKean-Vlasov SDEs driven by multiplicative Brownian noise, which is new even in the distribution independent case.
\item[(4)]   The uniform in time propagation of chaos in $L^1$-Wasserstein distance for path dependent mean field interacting system driven by multiplicative Brownian noise is offered, where the drifts are assumed to be partially dissipative.
\end{enumerate}

The paper is organized as follows: In Section 2, we study exponential contractivity in entropy-cost under the uniformly dissipative condition. The exponential contractivity and uniform in time propagation of chaos in $L^1$-Wasserstein distance under the partially dissipative condition are investigated in Section 3. Finally, the well-posedness for path dependent McKean-Vlasov SDEs and mean field interacting particle system as well as a uniform in time estimate for the second moment of  path dependent McKean-Vlasov SDEs is provided in Section 4. Throughout the paper, we will use $C$ or $c$ as a constant, the values of which may change from one place to another.
%In Section 2, we state the main results, i.e. the well-posedness and Log-Harnack inequality for path-dependent McKean-Vlasov SDEs with distribution dependent noise and the proof is provided in Section 3 and Section 4.

\section{Coupling by change of conditional probability measure and applications}
\subsection{Log-Harnack inequality}
To derive the log-Harnack inequality, we make the following assumptions.
\begin{enumerate}
	\item[\bf{(A)}] $\sigma(\xi,\gamma)$ only depends on $\gamma$.
	There exists a constant $K_0>0$ such that for any $\xi,\eta\in\C$ and $\mu,\nu\in\mathscr{P}_2(\C_\infty)$,
	\begin{equation*}\begin{aligned}
	&|B(\xi,\mu)-B(\eta,\nu)|\leq K_0\left(\|\xi-\eta\|_{\infty}+\W^{\C_\infty}_2(\mu,\nu)\right),\\
	&\|\sigma(\mu)-\sigma(\nu)\|_{HS}\leq K_0\W^{\C_\infty}_2(\mu,\nu).
	\end{aligned}\end{equation*}
\end{enumerate}
Recall $\pi_0(\xi)=\xi(0)$,\ \ $\xi\in\C$. For any $\gamma\in\scr P(\C_\infty)$, let $$(\gamma\circ \pi_0^{-1})(A)=\gamma(\{\xi:\pi_0(\xi)\in A\}), \ \ A\in\mathscr B(\R^d).$$
The following result characterizes the log-Harnack inequality for \eqref{EQ} under {\bf (A)} and the main tool is coupling by change of conditional probability measure.
\begin{thm}\label{RL} Assume {\bf (A)}. Then there exists an increasing function $c:(0,\infty)\to[0,\infty)$ such that for positive $f\in\mathscr{B}_b(\R^d)$,
	\beg{equation*}\beg{split}
	&((P_t^\ast\mu_0)\circ \pi_0^{-1})(\log f) \\
&\leq\log \{((P_t^\ast\nu_0)\circ \pi_0^{-1})(f)\}+\frac{c(t)}{t}\W^{\C_\infty}_2(\mu_0,\nu_0)^2,\ \ t> 0,  \mu_0,\nu_0\in \scr P_{2}(\C_\infty).\end{split}\end{equation*}
Consequently, it holds
\begin{align*}
\|(P_t^\ast\mu_0)\circ \pi_0^{-1}-(P_t^\ast\nu_0)\circ \pi_0^{-1}\|_{var}^2&\leq2\mathrm{Ent}((P_t^\ast\mu_0)\circ \pi_0^{-1}|(P_t^\ast\nu_0)\circ \pi_0^{-1})\\
&\leq \frac{2c(t)}{t}\W^{\C_\infty}_{2}(\mu_0,\nu_0)^2,\ \ t>0,  \mu_0,\nu_0\in \scr P_{2}(\C_\infty).
\end{align*}
\end{thm}
\begin{proof} Since $\sigma(\xi,\gamma)$ does not depend on $\xi$ due to {\bf(A)}, \eqref{EQ} reduces to
\beq\label{EQ0}\d X(t)=B(X_t,\L_{X_t})\d t+\beta\d W^1(t)+\sigma(\L_{X_t})\d W^2(t).\end{equation}
Let $\mu_0,\nu_0\in \scr P_{2}(\C_\infty)$. Denote $\mu_t=P_t^\ast\mu_0,\nu_t=P_t^\ast\nu_0$. In view of {\bf(A)}, it is standard to derive
\begin{align}\label{w2t}\W_2^{\scr C_\infty}(\mu_t,\nu_t)\leq C(t)\W_2^{\scr C_\infty}(\mu_0,\nu_0),\ \ t\geq 0
\end{align}
for some increasing function $C:[0,\infty)\to[0,\infty)$.
Let $X_0,\tilde{X}_0$ be two $\F_0$-measurable $\C_\infty$-valued random variables such that
$$\L_{X_0}=\mu_0,\ \ \L_{\tilde{X}_0}=\nu_0. $$ Let $X_t$ solve \eqref{EQ0} with initial value $X_0$, i.e.
\beq\label{EQ1}\d X(t)=B(X_t,\mu_t)\d t+\beta\d W^1(t)+\sigma(\mu_t)\d W^2(t).\end{equation}
Fix $t_0>0$. Let
\beq\label{XI}
\xi ^{\nu}(t):=\int_0^{t\lor 0}{\sigma(\nu_s)}\d W^2 (s),\ \ t\in [-r_0,t_0],
\end{equation}
and $\xi^\mu(t)$ be defined in the same way by replacing $\nu_t$ with $\mu_t$.
By {\bf(A)} and BDG's inequality, we arrive at
\beq\label{XX}\begin{aligned}
\E \sup_{t\in [0,t_0]}\left|\xi^{\mu}(t)-\xi^{\nu}(t)\right|^2\leq 4K_0^2\int_{0}^{t_0}\W^{\C_\infty}_2(\mu_t,\nu_t)^2\d t.
\end{aligned}\end{equation}
Consider
\beq\label{YEQ}\begin{aligned}
	\d Y(t)=&B(X_t,\mu_t)\d t+\frac{\xi^{\mu}(t_0)-\xi^{\nu}(t_0)+X_0(0)-\tilde{X}_0(0)}{t_0} \d t\\
	&+\beta\d W^1(t)+\sigma(\nu_t)\d W^2(t),\ \ t\in[0,t_0],Y_0=\tilde{X}_0.
\end{aligned}\end{equation}
From \eqref{EQ1}, \eqref{XI} and \eqref{YEQ}, we derive
\beq\label{YXEQ}\begin{aligned}
	Y(t)-X(t)&=\frac{t_0-t}{t_0}\left(\tilde{X}_0(0)-X_0(0)\right)\\
	& +\frac{t}{t_0}\left(\xi^{\mu}(t_0)-\xi^{\nu}(t_0)\right)+\xi^{\nu}(t)-\xi^{\mu}(t),\ \ t\in[0,t_0].
\end{aligned}\end{equation}
This implies that
\begin{equation}\label{Y-X}\begin{aligned}
\|Y_t-X_t\|^2_{\infty}&\leq\sup_{t\in[-r_0,t_0]}|Y(t)-X(t)|^2\\
&\leq4\|X_0-\tilde{X}_0\|_{\infty}^2+4\sup_{t\in[0,t_0]}|\xi^{\mu}(t)-\xi^{\nu}(t)|^2,\ \ t\in[0,t_0].
\end{aligned}\end{equation}
Let
\begin{equation*}\begin{aligned}
\phi(t):=&\frac{1}{\beta}\left[B(Y_t,\nu_t)-B(X_t,\mu_t)\right]\\
&-\frac{1}{\beta t_0}\left[\xi^{\mu}(t_0)-\xi^{\nu}(t_0)+X_0(0)-\tilde{X}_0(0)\right],\ \ t\in[0,t_0].
\end{aligned}\end{equation*}
According to {\bf(A)}, there exists a constant $C_1>0$ such that
\begin{align*}
	|\phi(t)|^2&\leq \frac{C_1}{t^2_0}\left[|X_0(0)-\tilde{X}_0(0)|^2+\sup_{t\in[0,t_0]}|\xi^{\nu^1}(t)-\xi^{\nu^2}(t)|^2\right]\\
	&\ \ \ \ + C_1\W^{\C_\infty}_2(\mu_t,\nu_t)^2+C_1 \|Y_t-X_t\|^2_{\infty},\ \ t\in[0,t_0].
\end{align*}
This together with \eqref{Y-X} yields that we can find a constant $C_2>0$ such that
\beq\label{IPH2}\begin{aligned}
	\frac{1}{2}\int_{0}^{t_0}|\phi(t)|^2\d t&\leq C_2\int_{0}^{t_0}\W^{\C_\infty}_2(\mu_t,\nu_t)^2\d t+C_2 t_0\|X_0-\tilde{X}_0\|_{\infty}^2\\
	& +\frac{C_2}{t_0}|X_0(0)-\tilde{X}_0(0)|^2+\left(\frac{C_2}{t_0}+C_2t_0\right)\sup_{t\in[0,t_0]}|\xi^{\mu}(t)-\xi^{\nu}(t)|^2.
\end{aligned}\end{equation}
We will use the conditional probability and the conditional expectation given by both $W^2$ and $\F_0$:
\begin{equation*}
	\P^{W^2,0}:=\P(\cdot|W^2,\F_0),\ \ \E^{W^2,0}:=\E(\cdot|W^2,\F_0).
\end{equation*}
Let $\d\Q^{W^2,0}:= R\d\P^{W^2,0}$, where
$$R:=\e^{\int_{0}^{t_0}\<\phi(s),\d W^1(s) \>-\frac{1}{2}\int_{0}^{t_0}|\phi(s)|^2\d s}.$$
By Girsanov's theorem, under the weighted conditional probability $\Q^{W^2,0}$,
$$\widetilde{W}^1(t):=W^1(t)-\int_{0}^{t}\phi(s)\d s,\ \ t\in[0,t_0]$$
is a $d$-dimensional Brownian motion. By \eqref{YEQ}, $\hat{Y}(t):=Y(t)-\xi^{\nu}(t),t\in[-r_0,t_0]$ solves the SDE
$$\d \hat{Y}(t)=B(\hat{Y}_t+\xi^{\nu}_t,\nu_t)\d t+\beta\d \widetilde{W}^1(t),\ \,\hat{Y}_0=\tilde{X}_0,\ \ t\in[0,t_0].$$
Let $\tilde{X}(t)$  solve
\begin{align*}\d \tilde{X}(t)=B(\tilde{X}_t,\nu_t)\d t+\beta\d W^1(t)+\sigma(\nu_t)\d W^2(t)\end{align*}
with initial value $\tilde{X}_0$.
Then
$$\hat{X}(t):=\tilde{X}(t)-\xi^{\nu}(t),\ \ t\in[-r_0,t_0]$$
solves
\begin{align}\label{tXi}\d \hat{X}(t)=B(\hat{X}_t+\xi^{\nu}_t,\nu_t)\d t+\beta\d W^1(t),\ \,\hat{X}_0=\tilde{X}_0,\ \ t\in[0,t_0].
\end{align}
The weak uniqueness of \eqref{tXi} gives
$$\L_{\hat{Y}(t_0)|\Q^{W^2,0}}=\L_{\hat{X}(t_0)|\P^{W^2,0}}.$$
Since $\xi^{\nu}(t_0)$ is deterministic given $W^2$, it follows that
$$\L_{Y(t_0)|\Q^{W^2,0}}=\L_{[\hat{Y}(t_0)+\xi^{\nu}(t_0)]|\Q^{W^2,0}}=\L_{[\hat{X}(t_0)+\xi^{\nu}(t_0)]|\P^{W^2,0}} =\L_{\tilde{X}(t_0)|\P^{W^2,0}}.$$
Combining this with the fact $X(t_0)=Y(t_0)$ due to \eqref{YXEQ}, we obtain
\begin{equation*}
\E^{W^2,0}[f(\tilde{X}(t_0))]=\E_{\Q^{W^2,0}}[f(Y(t_0))]=\E^{W^2,0}[R f(X(t_0))],\ \ f\in\B_b(\R^d).
\end{equation*}
By Young's inequality \cite[Lemma 2.4]{MAW}, for any $ f\in \mathscr{B}_b(\R^d)$ with $f>0$, we can derive
\beq\label{YOUNG}\begin{aligned}
	&\E^{W^2,0}\left[\log f(\tilde{X}(t_0))\right]
	&\leq \log \E^{W^2,0}\left[f(X(t_0))\right]+\E^{W^2,0}\left[R \log R\right],
\end{aligned}\end{equation}
and
\beq\label{YOUNG2}\begin{aligned}
	&\ \ \ \ \E^{W^2,0}\left[R \log R\right]\\
	&=\E_{\Q^{W^2,0}}\left[\int_{0}^{t_0}\left\<\phi(s),\d \left(W^1(s) -\int_{0}^{s}\phi(r)\d r\right)\right\>\right] +\E_{\Q^{W^2,0}}\left[\frac{1}{2}\int_{0}^{t_0}|\phi(s)|^2\d s\right]\\
	&=\E_{\Q^{W^2,0}}\left[\frac{1}{2}\int_{0}^{t_0}|\phi(s)|^2\d s\right].
\end{aligned}\end{equation}
Combining \eqref{YOUNG} with \eqref{YOUNG2}, \eqref{IPH2}, \eqref{XX}, \eqref{YOUNG2} and Jensen's inequality, we find a constant $C_3>0$ such that for $f>0$ and $f\in \mathscr{B}_b(\R^d)$,
\begin{equation*}\label{PTLF0}\begin{aligned}
&\ \ \ \ \E \left[\log f\left(\tilde{X}(t_0)\right)\right]\\
&=\E\left[\E^{W^2,0} \left(\log f\left(\tilde{X}(t_0)\right)\right)\right]\\
&\leq \E\left[\log \E^{W^2,0} f\left(X(t_0)\right)\right]+\E\left[\E_{\Q^{W^2,0}}\left[\frac{1}{2}\int_{0}^{t_0}|\phi(s)|^2\d s\right]\right]\\ 	
&\leq \log \E f\left(X(t_0)\right) +C_3t_0\E\|X_0-\tilde{X}_0\|^2_{\infty}+\frac{C_3}{t_0}\E|X_0(0)-\tilde{X}_0(0)|^2\\
&+ {C_3}\left(1+t_0+\frac{1}{t_0}\right)\int_{0}^{t_0}\W^{\C_\infty}_2(\mu_t,\nu_t)^2\d t.
\end{aligned}\end{equation*}	
Taking the infimum with respect to all $X_0, \tilde{X_0}$ satisfying $\L_{X_0}=\mu_0,~\L_{\tilde{X_0}}=\nu_0$, and combining with \eqref{w2t}, we arrive at
\beg{equation*}\beg{split}
&	((P_{t_0}^\ast\nu_0)\circ \pi_0^{-1})(\log f) \\
&\leq\log \{((P_{t_0}^\ast\mu_0)\circ \pi_0^{-1})(f)\}+\frac{c(t_0)}{t_0}\W^{\C_\infty}_2(\mu_0,\nu_0)^2,\ \   \mu_0,\nu_0\in \scr P_{2}(\C_\infty)\end{split}\end{equation*}
for some increasing function $c:[0,\infty)\to[0,\infty)$.
According to \cite[Theorem 1.4.2(2)]{HARNACK}, we conclude
\begin{equation}\label{haha}\begin{aligned}
\Ent\left((P_t^\ast\nu_0)\circ \pi_0^{-1}|(P_t^\ast\mu_0)\circ \pi_0^{-1}\right)\leq\frac{c(t)}{t}\W^{\C_\infty}_2(\mu_0,\nu_0)^2,
\end{aligned}\end{equation}
Finally, by Pinsker's inequality \eqref{ETX}, we complete the proof.

\end{proof}	
\subsection{Exponential contractivity in entropy-cost}	
In this part, we will use the log-Harnack inequality to study the exponential contractivity in entropy-cost. To this end, we still assume that $\sigma(\xi,\gamma)$ only depends on $\gamma$ and consider the McKean-Vlasov SDEs:
\beq\label{EQth}\d X(t)=B(X_t,\L_{X_t})\d t+\beta\d W^1(t)+\sigma(\L_{X_t})\d W^2(t).\end{equation}
We make the following uniformly dissipative assumption.
\begin{enumerate}
	\item[\bf{(B)}] $B$ is continuous in $\C_\infty\times \mathscr{P}_2(\C_\infty)$ and there exist constants $K_1, K_2, K_3 >0$ such that for any $\xi,\eta\in\C$ and $\mu,\nu\in\mathscr{P}_2(\C_\infty)$,
\begin{equation*}\begin{aligned}
	&\ \ \ \ 2 \<B(\xi,\mu)-B(\eta,\nu), \xi(0)-\eta(0) \> \\
&\leq \frac{K_1}{2}\W^{\C_\infty}_2(\mu,\nu)^2 + K_2\|\xi-\eta\|_{\C_\infty}^2 - K_3|\xi(0)-\eta(0)|^2,
	\end{aligned}\end{equation*}
and
\begin{equation*}\begin{aligned}
	&\|\sigma(\mu)-\sigma(\nu)\|_{HS}^2\leq\frac{ K_1}{2}\W^{\C_\infty}_2(\mu,\nu)^2 .
	\end{aligned}\end{equation*}
\end{enumerate}
\begin{rem}\label{wellp} According to Theorem \ref{wep} below, \eqref{EQth} is well-posed in $\scr P_2(\C_\infty)$.
\end{rem}
\begin{thm}\label{EE} Assume {\bf (B)} with
	\begin{align}\label{dic}20 K_1 + 2 K_2 < K_3 \e^{-K_3 r_0}.\end{align}
Then the following assertions hold.
\begin{enumerate}
\item[(1)] There exist constants $c, \lambda > 0$ such that
	\beg{equation*}\beg{split}
\W_2^{\C_\infty}(P^{\ast}_t\mu_0,P^{\ast}_t\nu_0)\leq c\e^{-\lambda t}\W_2^{\C_\infty}(\mu_0,\nu_0),\ \ t\geq 0, \mu_0,\nu_0\in \scr P_{2}(\C_\infty).
\end{split}\end{equation*}
\item[(2)] If in addition {\bf (A)} holds, then there exists a constant $\tilde{c}>0$ such that
	\begin{align*}
&\Ent\left((P_t^\ast\mu_0)\circ\pi_0^{-1}|(P_t^\ast\nu_0)\circ\pi_0^{-1}\right)\leq \tilde{c}\e^{-2\lambda t}\W_2^{\C_\infty}(\mu_0,\nu_0)^2,\ \ t\geq 1, \mu_0,\nu_0\in \scr P_{2}(\C_\infty).
\end{align*}
\end{enumerate}
\end{thm}
\begin{proof}(1) We will use the synchronous coupling to derive the exponential contractivity in $\W_2^{\C_\infty}$.
Let $\mu_0,\nu_0\in \scr P_{2}(\C_\infty)$ and $X_t$ and $Y_t$ solve \eqref{EQth} with initial values $X_0$ and $Y_0$ satisfying
\begin{align*}\L_{X_0} = \mu_0,~ \L_{Y_0} = \nu_0.\end{align*}
For simplicity, let $\mu_t = P^{\ast}_t\mu_0 = \L_{X_t}, \nu_t = P^{\ast}_t\nu_0 = \L_{Y_t}$. It follows from It\^o's formula and {\bf (B)} that
\begin{equation}\begin{aligned} \label{eq51}
	&\ \ \ \ \d |X(t) - Y(t)|^2\\
&\leq\left\{2 \<B(X_t,\mu_t)-B(Y_t,\nu_t), X(t)-Y(t) \> + \|\sigma(\mu_t)-\sigma(\nu_t)\|_{HS}^2\right\}\d t+ \d M(t)\\
&\leq \left\{K_1\W^{\C_\infty}_2(\mu_t,\nu_t)^2 + K_2\|X_t-Y_t\|_{\infty}^2 - K_3|X(t)-Y(t)|^2\right\}\d t+ \d M(t),\ \ t\geq 0,
\end{aligned}\end{equation} 	
where 	
	\begin{equation} \label{eq52}	
	\d M(t)= 2 \int_{0}^{t}\left\< X(s)-Y(s), \{\sigma(\mu_s)-\sigma(\nu_s)\}\d W^2(s)\right\>.
\end{equation}
Then it follows from  \eqref{eq51} that
\begin{equation}\begin{aligned} \label{eq53}
&\ \ \ \ \d\left[\e^{K_3 t}|X(t) - Y(t)|^2\right] \\
&\leq K_1\e^{K_3 t}\W^{\C_\infty}_2(\mu_t,\nu_t)^2\d t + K_2\e^{K_3 t}\|X_t-Y_t\|_{\infty}^2 \d t + \e^{K_3 t}\d M(t),\ \ t\geq 0.
\end{aligned}\end{equation}
Let $\eta_t= \sup_{s\in[-r_0,t]}\e^{K_3 s^{+}}|X(s) - Y(s)|^2$.	We conclude from \eqref{eq53} that
\begin{equation}\begin{aligned} \label{eq54}
\eta_t \leq &\|X_0-Y_0\|_{\infty}^2 + \int_{0}^{t}K_1\e^{K_3 r}\W^{\C_\infty}_2(\mu_r,\nu_r)^2\d r\\
&+\int_{0}^{t} K_2\e^{K_3 r}\|X_r-Y_r\|_{\infty}^2 \d r+\sup_{s\in[0,t]}\int_{0}^{s}\e^{K_3 r}\d M(r),\ \ t\geq 0.
\end{aligned}\end{equation}
By BDG's inequality, \eqref{eq52} and {\bf (B)}, we derive
	\begin{align*}
	&\ \ \ \ \E \sup_{s\in[0,t]}\int_{0}^{s}\e^{K_3 r}\d M(r)\\
	&\leq 6 \E\left\{\int_{0}^{t}\e^{2 K_3 r}\|\sigma(\mu_r)-\sigma(\nu_r)\|_{HS}^2\times|X(r)-Y(r)|^2 \d r \right\}^{\frac{1}{2}}\\
		&\leq \frac{1}{2}\E \eta_t + 9 \int_{0}^{t}K_1 \e^{ K_3 r}\W^{\C_\infty}_2(\mu_r,\nu_r)^2\d r.
	\end{align*}
Combining this with \eqref{eq54} and the fact $\eta_t\geq \e^{K_3 (t-r_0)}\|X_t-Y_t\|_{\infty}^2$, we obtain
	\begin{align*}
	&\ \ \ \ \E \e^{K_3 (t-r_0)}\|X_t-Y_t\|_{\infty}^2\\
	&\leq 2\E\|X_0-Y_0\|_{\infty}^2  +\int_{0}^{t}(2K_2+20K_1)\e^{K_3 r}\E \|X_r-Y_r\|_{\infty}^2\d r,\ \ t\geq 0.
	\end{align*}
By Gronwall's inequality, we have
\begin{align*}
\E \e^{K_3 (t-r_0)}\|X_t-Y_t\|_{\infty}^2\leq 2 \e^{\e^{K_3 r_0}(2K_2+20K_1)t}\E\|X_0-Y_0\|_{\infty}^2 ,\ \ t\geq 0.
\end{align*} 	
So, we arrive at
	\begin{equation}\begin{aligned} \label{eq58}
\E\|X_t-Y_t\|_{\infty}^2\leq 2\e^{K_3 r_0}\e^{\e^{K_3 r_0}(2K_2+20K_1-K_3\e^{-K_3 r_0})t}\E\|X_0-Y_0\|_{\infty}^2,\ \ t\geq 0.
	\end{aligned}\end{equation}
Taking infimum with respect to $X_0$ and $Y_0$ satisfying $\L_{X_0} = \mu_0,~ \L_{Y_0} = \nu_0$, by \eqref{dic} and \eqref{eq58}, we conclude that  there exist constants $c, \lambda > 0$ such that
\begin{align}\label{Cot}\W^{\C_\infty}_2(P_t^\ast\mu_0,P_t^\ast \nu_0) \leq c\e^{-\lambda t}\W^{\C_\infty}_2(\mu_0,\nu_0),\ \ t\geq 0.
\end{align}	
(2) By the log-Harnack ineqaulity \eqref{haha} for $t=1$, \eqref{Cot} and using the semigroup property $P^{*}_t=P^{*}_{1}P^{*}_{t-1}$ for $t\geq1$, there exists a constant $c_0>0$ such that for any $t\geq 1$,
\begin{align*}
&\Ent\left((P_t^\ast\mu_0)\circ\pi_0^{-1}|(P_t^\ast\nu_0)\circ\pi_0^{-1}\right)\\
&=\Ent\left((P_1^\ast P_{t-1}^\ast\mu_0)\circ\pi_0^{-1}|(P_1^\ast P_{t-1}^\ast\nu_0)\circ\pi_0^{-1}\right)\\
&\leq c_0\W^{\C_\infty}_2(P_{t-1}^\ast\mu_0,P_{t-1}^\ast\nu_0)^2\leq c_0 c^2\e^{-2\lambda (t-1)}\W_2^{\C_\infty}(\mu_0,\nu_0)^2,\ \ \mu_0,\nu_0\in \scr P_{2}(\C_\infty).
\end{align*}
Therefore, the proof is completed.
\end{proof}

\section{Asymptotic reflection coupling and applications in path dependent system}
As explained in Introduction, to derive the exponential contractivity in $L^1$-Wasserstein distance under partially dissipative condition, we will replace the uniform norm in $\C$ with some integral type norm. More precisely, let
\begin{align}\label{gamma}\Gamma^{r_0}(\d s)=\frac{1}{2}r_0^{-1}\d s+\frac{1}{2}\delta_0(\d s), \ \ s\in[-r_0,0],
\end{align}
 here $\delta_0$ is the Dirac measure at point $0$. $\Gamma^{r_0}$ is a probability measure on $[-r_0,0]$ and $r_0^{-1}1_{[-r_0,0]}(s)\d s$ converges to $\delta_0$ weakly as $r_0\to 0$ so that we use the convention $\Gamma^{0}(\d s)=\delta_0$. Define
\begin{align}\label{xga}\|\xi\|_{\Gamma^{r_0}}=\int_{[-r_0,0]}|\xi(s)|\Gamma^{r_0}(\d s)=\frac{1}{2}|\xi(0)|+\frac{1}{2}r_0^{-1}\int_{[-r_0,0]}|\xi(s)|\d s,\ \ \xi\in\C.
\end{align}
Define the $L^1$-Wasserstein distance induced by $\|\cdot\|_{\Gamma^{r_0}}$:
$$\W_1^{\Gamma^{r_0}}(\mu,\nu):= \inf_{\pi\in\textbf{C}(\mu,\nu)}\int_{\mathscr{C}\times \C}\|\xi-\tilde{\xi}\|_{\Gamma^{r_0}}\pi(\d \xi,\d \tilde{\xi}),\ \ \mu,\nu\in\scr P_1(\C_\infty),$$
where $\textbf{C}(\mu,\nu)$ is the set of all couplings for $\mu$ and $\nu$.
It is not difficult to see that
\begin{align}\label{1in0}\|\xi-\eta\|_{\Gamma^{r_0}}\leq \|\xi-\eta\|_\infty,\ \ \xi,\eta\in\C,
\end{align}
and hence
\begin{align}\label{1in}\W_1^{\Gamma^{r_0}}(\mu,\nu)\leq \W_1^{\C_\infty}(\mu,\nu),\ \ \mu,\nu\in\scr P_1(\C_\infty).
\end{align}
In this section, we assume $\sigma(\xi,\mu)=\sigma(\xi(0))$ and
$B(\xi,\mu)=b^{(0)}(\xi(0))+b^{(1)}(\xi,\mu)$ with  $b^{(0)}:\R^d\to\R^d$ and $b^{(1)}:\C\times \scr P_1(\C_\infty)\to\R^d$.
\subsection{Exponential contractivity in $L^1$-Wasserstein distance for path dependent McKean-Vlasov SDEs}
Consider
\begin{align} \label{al10}
\d X(t)&=b^{(0)}(X(t))\d t+b^{(1)}(X_t,\L_{X_t})\d t+\beta\d W^1(t)+\sigma(X(t))\d W^2(t).
\end{align}

To derive the exponential contractivity in $\W_1^{\Gamma^{r_0}}$, we make the following assumptions.
\begin{enumerate}
\item[{\bf(C)}] There exists a constant $K_\sigma>0$ such that
\begin{align}\label{bslipsg}
&\frac{1}{2}\|\sigma(x_1)-\sigma(x_2)\|^2_{HS}\leq K_\sigma|x_1-x_2|^2, \ \ x_1,x_2\in\R^d.
\end{align}
$b^{(0)}$ is continuous and there exist $R>0$, $K_1\geq 0, K_2>K_\sigma$  such that
\begin{align}\label{pdi}
&\langle x_1-x_2, b^{(0)}(x_1)-b^{(0)}(x_2)\rangle\leq \gamma(|x_1-x_2|)|x_1-x_2|,\ \ x_1,x_2\in\R^d
\end{align}
with
$$\gamma(r)=\left\{
  \begin{array}{ll}
K_1r, & \hbox{$r\leq R$;} \\
    \left\{-\frac{K_1+K_2}{R}(r-R)+K_1\right\}r, & \hbox{$R\leq r\leq 2R$;} \\
    -K_2r, & \hbox{$r>2R$.}
  \end{array}
\right.
$$
Moreover, there exists $K_b\geq0$ such that
\begin{align}\label{lip1a}\nonumber&|b^{(1)}(\xi,\mu)-b^{(1)}(\tilde{\xi},\tilde{\mu})|\\
&\leq K_b(\|\xi-\tilde{\xi}\|_{\Gamma^{r_0}}+\W_1^{\Gamma^{r_0}}(\mu,\tilde{\mu})),\ \ \xi,\tilde{\xi}\in\scr C,\mu,\tilde{\mu}\in\scr P_1(\C_\infty).
\end{align}
\end{enumerate}
\begin{rem}\label{mtd} Again by Theorem \ref{wep} below, under {\bf (C)} and combining with \eqref{1in0} and \eqref{1in}, \eqref{al10} is well-posed in $\scr P_1(\C_\infty)$.
By \eqref{pdi}, $b^{(0)}$ is partially dissipative, i.e. dissipative in long distance and it is equivalent to the condition that there exist $C_1\geq0, C_2>0,D>0$ such that for any $x_1,x_2\in\R^d$,
\begin{align*}
&\langle x_1-x_2, b^{(0)}(x_1)-b^{(0)}(x_2)\rangle\leq C_1|x_1-x_2|^2 1_{\{|x_1-x_2|\leq D\}}-C_2|x_1-x_2|^21_{\{|x_1-x_2|>D\}}.
\end{align*}
For example, let $V(x)=-|x|^4+|x|^2$, $b^{(0)}(x)=\nabla V(x)=-4|x|^2x+2x$, then $b^{(0)}$ satisfies \eqref{pdi}.
\end{rem}
The main result in this part is the following theorem.
\begin{thm}\label{POC10}
Assume {\bf(C)} with
\begin{align}\label{kb-kd1}
K_b<\frac{2\beta^4}{\e^{\frac{2\beta^2}{\delta}r_0}(K_2-K_\sigma)\delta^2}
 \end{align}
 for $\delta:=\int_0^\infty s\e^{\frac{1}{2\beta^2}\int_0^s\{\gamma(v)+K_\sigma v\}\d v}\d s$.
Then there exist constants $c,\lambda>0$ such that
\begin{align*}&\W_1^{\Gamma^{r_0}}(P_t^\ast\mu_0,P_t^\ast\nu_0)\leq c\e^{-\lambda t}\W_1^{\Gamma^{r_0}}(\mu_0,\nu_0),\ \ \mu_0,\nu_0\in \scr P_{1}(\C_\infty).
\end{align*}
\end{thm}
\begin{rem} When $r_0=0$, \eqref{kb-kd1} reduces to
$$K_b<\frac{2\beta^4}{(K_2-K_\sigma)\delta^2}.$$
\end{rem}
%(2) If in particular, $b_t(x,\mu)=\int_{\R^d}\tilde{b}(x-y)\mu(\d x)$ for some Lipschitz continuous function  $\tilde{b}$, then for any $X_0^i$ with $\L_{X_0^i}\in\scr P_2$, the assertion in (1) holds for
%$\frac{1}{N}$ replacing $R_{d,q}(N)$.

\begin{proof}

 %For $\mu_0,\nu_0\in \scr P_1(\C_\infty)$, take $\F_0$-measurable and $\C_\infty$-valued random variables $X_0$ and $\tilde{X}_0$ such that $$\L_{X_0}=\mu_0,\ \ \L_{\tilde{X}_0}=\nu_0.$$
We adopt the technique of asymptotic reflection coupling. Let $\mu_0,\nu_0\in \scr P_1(\C_\infty)$ and denote $\mu_t=P_t^\ast\mu_0, \nu_t=P_t^\ast\nu_0$. For any $\varepsilon\in(0,1]$, let $\pi_R^\varepsilon\in[0,1]$ and $\pi_S^\varepsilon$ be two Lipschitz continuous functions on $[0,\infty)$ satisfying
\begin{align}\label{pir}\pi_R^\varepsilon(x)=\left\{
      \begin{array}{ll}
        1, & \hbox{$x\geq \varepsilon$;} \\
        0, & \hbox{$x\leq \frac{\varepsilon}{2}$}
      \end{array}
    \right.,\ \ (\pi_R^\varepsilon)^2+(\pi_S^\varepsilon)^2=1.
\end{align}
Let $\{\tilde{W}^1(t)\}_{t\geq 0}$ be a $d$-dimensional Brownian motion independent of $\{W^1(t),W^2(t)\}_{t\geq 0}$. Construct
\begin{equation*}\begin{split}
\d \tilde{X}(t)&=b^{(0)}(\tilde{X}(t))\d t+b^{(1)}(\tilde{X}_t,\mu_t)\d t+\sigma(\tilde{X}(t))\d W^2(t)\\
&+\beta\pi_R^\varepsilon(|\tilde{Z}(t)|)\d
W^1(t)+\beta\pi_S^\varepsilon(|\tilde{Z}(t)|)\d \tilde{W}^1(t),
\end{split}\end{equation*}
and
\begin{equation*}\begin{split}
\d Y(t)&=b^{(0)}(Y(t))\d t+b^{(1)}(Y_t,\nu_t)\d t+\sigma(Y(t))\d W^2(t)\\
&+\beta\pi_R^\varepsilon(|\tilde{Z}(t)|)(I_{d\times d}-2\tilde{U}_t\otimes \tilde{U}_t)\d
W^1(t)+\beta\pi_S^\varepsilon(|\tilde{Z}(t)|)\d \tilde{W}^1(t),
\end{split}\end{equation*}
where  $\tilde{Z}(t)=\tilde{X}(t)-Y(t)$, $\tilde{U}_t=\frac{\tilde{Z}(t)}{|\tilde{Z}(t)|}1_{\{|\tilde{Z}(t)|\neq 0\}}$ and $\L_{\tilde{X}_0}=\mu_0$ and $\L_{Y_0}=\nu_0$.
By the It\^{o}-Tanaka formula, \eqref{bslipsg}, \eqref{pdi} and \eqref{lip1a}, we arrive at
\begin{align*}
\d |\tilde{Z}(t)|&\leq \left\<b^{0}(\tilde{X}(t))-b^{0}(Y(t)),\frac{\tilde{Z}(t)} {|\tilde{Z}(t)|}1_{\{|\tilde{Z}(t)|\neq 0\}}\right\>\d t \\ &+\frac{1}{2}\|\sigma(\tilde{X}(t))-\sigma(Y(t))\|_{HS}^2\frac{1}{|\tilde{Z}(t)|}1_{\{|\tilde{Z}(t)|\neq 0\}}\d t+K_b\|\tilde{Z}_t\|_{\Gamma^{r_0}}\d t+K_b\W_1^{\Gamma^{r_0}}(\mu_t,\nu_t)\d t\\
&+\left\<[\sigma(\tilde{X}(t))-\sigma(Y(t))]\d W^2(t),\frac{\tilde{Z}(t)}{|\tilde{Z}(t)|}1_{\{|\tilde{Z}(t)|\neq 0\}}\right\>\\
&+ 2\beta\pi_R^\varepsilon(|\tilde{Z}(t)|)\left\<\frac{\tilde{Z}(t)}{|\tilde{Z}(t)|}1_{\{|\tilde{Z}(t)|\neq 0\}},\d W^1(t)\right\>\\
&\leq \gamma(|\tilde{Z}(t)|)\d t+K_b\|\tilde{Z}_t\|_{\Gamma^{r_0}}\d t+K_\sigma|\tilde{Z}(t)|\d t+K_b\W_1^{\Gamma^{r_0}}(\mu_t,\nu_t)\d t\\
&+\left\<[\sigma(\tilde{X}(t))-\sigma(Y(t))]\d W^2(t),\frac{\tilde{Z}(t)}{|\tilde{Z}(t)|}1_{\{|\tilde{Z}(t)|\neq 0\}}\right\>\\
&+ 2\beta\pi_R^\varepsilon(|\tilde{Z}(t)|)\left\<\frac{\tilde{Z}(t)}{|\tilde{Z}(t)|}1_{\{|\tilde{Z}(t)|\neq 0\}},\d W^1(t)\right\>.
\end{align*}
Let $$\tilde{\gamma}(v)=\gamma(v)+K_{\sigma}v,\ \ v\geq 0,$$ and define
$$f(r)=\int_0^r\e^{-\frac{1}{2\beta^2}\int_0^u\tilde{\gamma}(v)\d v}\int_u^\infty s\e^{\frac{1}{2\beta^2}\int_0^s\tilde{\gamma}(v)\d v}\d s\d u,\ \ r\geq 0.$$
Then $f$ enjoys the following properties:
$$f'(0)=\delta=\int_0^\infty s\e^{\frac{1}{2\beta^2}\int_0^s\tilde{\gamma}(v)\d v}\d s,$$
\begin{align}\label{fii}f''(r)\leq 0, \ \ r\geq 0,
\end{align}
\begin{align*}\lim_{r\to\infty}\frac{f(r)}{r}=\lim_{r\to\infty}f'(r)=\lim_{r\to\infty}\frac{2\beta^2}{-\tilde{\gamma}'(r)} =\frac{2\beta^2}{K_2-K_\sigma},
\end{align*}
\begin{align}\label{see}f'(r)\tilde{\gamma}(r)+2\beta^2 f''(r)= -2\beta^2 r,\ \ r\geq 0,
\end{align}
and
\begin{align}\label{cop}\frac{2\beta^2}{K_2-K_\sigma}r\leq f(r)\leq f'(0)r,
\end{align}
see \cite[Page 1054]{W23} for more details.
%Therefore,
%$$\frac{1}{f'(0)}\W_f\leq \W_1(\mu,\nu)\leq \frac{K_2-K_b}{2\beta}\W_f. $$
By It\^{o}'s formula and \eqref{fii}, we have
\begin{align}\label{itf}
\nonumber\d f(|\tilde{Z}(t)|)&\leq f'(|\tilde{Z}(t)|)K_b\|\tilde{Z}_t\|_{\Gamma^{r_0}}\d t+f'(|\tilde{Z}(t)|)K_b\W_1^{\Gamma^{r_0}}(\mu_t,\nu_t)\d t\\
&+f'(|\tilde{Z}(t)|)\left\<[\sigma(\tilde{X}(t))-\sigma(Y(t))]\d W^2(t),\frac{\tilde{Z}(t)}{|\tilde{Z}(t)|}1_{\{|\tilde{Z}(t)|\neq 0\}}\right\>\\
\nonumber&+f'(|\tilde{Z}(t)|) 2\beta\pi_R^\varepsilon(|\tilde{Z}(t)|)\left\<\frac{\tilde{Z}(t)}{|\tilde{Z}(t)|}1_{\{|\tilde{Z}(t)|\neq 0\}},\d W^1(t)\right\>\\
\nonumber&+f'(|\tilde{Z}(t)|)\tilde{\gamma}(|\tilde{Z}(t)|)\d t+2\beta^2 f''(|\tilde{Z}(t)|)\pi_R^\varepsilon(|\tilde{Z}(t)|)^2\d t.
\end{align}
It follows from $\|f'\|_\infty=f'(0)$, \eqref{pir} and \eqref{see} that
\begin{align}\label{MKG}
				\nonumber&f'(|\tilde{Z}(t)|)\tilde{\gamma}(|\tilde{Z}(t)|)+2\beta^2 f''(|\tilde{Z}(t)|)\pi_R^\varepsilon(|\tilde{Z}(t)|)^2\\
				\nonumber&\leq \left(f'(|\tilde{Z}(t)|)\tilde{\gamma}(|\tilde{Z}(t)|)+2\beta^2 f''(|\tilde{Z}(t)|)\right)\pi_R^\varepsilon(|\tilde{Z}(t)|)^2\\ &+\|f'\|_\infty\left\{\sup_{s\in[0,\varepsilon]}\gamma^{+}(s)+K_\sigma\varepsilon\right\}\\
				\nonumber&\leq -2\beta^2 |\tilde{Z}(t)|+2\beta ^2|\tilde{Z}(t)|\left(1-\pi_R^\varepsilon(|\tilde{Z}(t)|)^2\right) +\|f'\|_\infty\left\{\sup_{s\in[0,\varepsilon]}\gamma^{+}(s)+K_\sigma\varepsilon\right\}\\
				\nonumber&\leq -2\beta^2 |\tilde{Z}(t)|+\ell(\varepsilon)
			\end{align}
for
\begin{align}\label{lva}\ell(\varepsilon):=2\beta^2 \varepsilon +f'(0)\left\{\sup_{s\in[0,\varepsilon]}\gamma^{+}(s)+K_\sigma\varepsilon\right\}.
\end{align}
This combined with \eqref{cop} and \eqref{itf} gives
\begin{align*}
\d f(|\tilde{Z}(t)|)
&\leq -\frac{2\beta^2}{f'(0)}f(|\tilde{Z}(t)|)\d t+f'(0)K_b\|\tilde{Z}_t\|_{\Gamma^{r_0}}\d t+f'(0)K_b\W_1^{\Gamma^{r_0}}(\mu_t,\nu_t)\d t+\ell(\varepsilon)\d t\\
&+f'(|\tilde{Z}(t)|)\left\<[\sigma(\tilde{X}(t))-\sigma(Y(t))]\d W^2(t),\frac{\tilde{Z}(t)}{|\tilde{Z}(t)|}1_{\{|\tilde{Z}(t)|\neq 0\}}\right\>\\
\nonumber&+f'(|\tilde{Z}(t)|) 2\beta\pi_R^\varepsilon(|\tilde{Z}(t)|)\left\<\frac{\tilde{Z}(t)}{|\tilde{Z}(t)|}1_{\{|\tilde{Z}(t)|\neq 0\}},\d W^1(t)\right\>.
\end{align*}
Let $\lambda_0=\frac{2\beta^2}{f'(0)}$. It follows that
\begin{align}\label{fzn}
\nonumber\d [\e^{\lambda_0 t}f(|\tilde{Z}(t)|)]&\leq \e^{\lambda_0 t}f'(0)K_b\|\tilde{Z}_t\|_{\Gamma^{r_0}}\d t+\e^{\lambda_0 t}f'(0)K_b\W_1^{\Gamma^{r_0}}(\mu_t,\nu_t)\d t\\
&+\e^{\lambda_0 t}\ell(\varepsilon)\d t+\d M_t,
\end{align}
where
\begin{align*}M_v&=\int_0^{v}\e^{\lambda_0 t}f'(|\tilde{Z}(t)|)\left\<[\sigma(\tilde{X}(t))-\sigma(Y(t))]\d W^2(t),\frac{\tilde{Z}(t)}{|\tilde{Z}(t)|}1_{\{|\tilde{Z}(t)|\neq 0\}}\right\>\\
&+\int_0^{v}\e^{\lambda_0 t}f'(|\tilde{Z}(t)|) 2\beta\pi_R^\varepsilon(|\tilde{Z}(t)|)\left\<\frac{\tilde{Z}(t)}{|\tilde{Z}(t)|}1_{\{|\tilde{Z}(t)|\neq 0\}},\d W^1(t)\right\>,\ \ v\geq 0.
\end{align*}
For $v\geq 0$, we derive from \eqref{fzn} that
\begin{align}\label{fzndv}
\nonumber\e^{\lambda_0 v}f(|\tilde{Z}(v)|)
&\leq f(|\tilde{Z}(0)|)+\int_0^{v}\e^{\lambda_0 t}f'(0)K_b\|\tilde{Z}_t\|_{\Gamma^{r_0}}\d t+\int_0^{v}\e^{\lambda_0 t}f'(0)K_b\W_1^{\Gamma^{r_0}}(\mu_t,\nu_t)\d t\\
&+\int_0^{v}\e^{\lambda_0 t}\ell(\varepsilon)\d t+M_v.
\end{align}
So, it follows from \eqref{fzndv} that for any $s\geq 0$,
\begin{align}\label{fzn12}
\nonumber&\frac{1}{r_0}\int_{(-r_0+s)\vee0}^s\e^{\lambda_0 v}f(|\tilde{Z}(v)|)\d v\\
\nonumber&\leq \frac{1}{r_0}\int_{(-r_0+s)\vee0}^sf(|\tilde{Z}(0)|)\d v+f'(0)K_b\frac{1}{r_0}\int_{(-r_0+s)\vee0}^s\int_0^{v}\e^{\lambda_0 t}\|\tilde{Z}_t\|_{\Gamma^{r_0}}\d t\d v\\
&+f'(0)K_b\frac{1}{r_0}\int_{(-r_0+s)\vee0}^s\int_0^{v}\e^{\lambda_0 t}\W_1^{{\Gamma^{r_0}}}(\mu_t,\nu_t)\d t\d v\\
\nonumber&+\frac{1}{r_0}\int_{(-r_0+s)\vee0}^s\int_0^{v}\e^{\lambda_0 t}\ell(\varepsilon) \d t\d v+\frac{1}{r_0}\int_{(-r_0+s)\vee0}^sM_v\d v\\
\nonumber&\leq f(|\tilde{Z}(0)|)+f'(0)K_b\int_0^{s}\e^{\lambda_0 t}\|\tilde{Z}_t\|_{\Gamma^{r_0}}\d t\\
\nonumber&+f'(0)K_b\int_0^{s}\e^{\lambda_0 t}\W_1^{\Gamma^{r_0}}(\mu_t,\nu_t)\d t+\int_0^{s}\e^{\lambda_0 t}\ell(\varepsilon) \d t+\frac{1}{r_0}\int_{(-r_0+s)\vee0}^sM_v\d v.
\end{align}
Note that
$$\e^{\lambda_0 (s-r_0)}\frac{1}{r_0}\int_{-r_0}^0f(|\tilde{Z}(s+u)|)\d u\leq\frac{1}{r_0}\int_{-r_0}^0\e^{\lambda_0 (s+u)^+}f(|\tilde{Z}(s+u)|)\d u,\ \ s\geq 0.$$
This together with \eqref{gamma}, \eqref{xga}, \eqref{cop} and the translation invariance of Lebesgue measure implies
\begin{align*}
\e^{\lambda_0 s}\|\tilde{Z}_s\|_{\Gamma^{r_0}}&=\frac{1}{2}
\e^{\lambda_0 s}\frac{1}{r_0}\int_{-r_0}^0|\tilde{Z}(s+u)|\d u+\frac{1}{2}\e^{\lambda_0 s}|\tilde{Z}(s)|\\
&\leq\frac{1}{2}\e^{\lambda_0 r_0}\frac{1}{r_0}\int_{-r_0}^0\e^{\lambda_0 (s+u)^+}|\tilde{Z}(s+u)|\d u+\frac{1}{2}\e^{\lambda_0 s}|\tilde{Z}(s)|\\
&\leq  \frac{1}{2}\e^{\lambda_0r_0}\frac{1}{r_0}\int_{-r_0}^0|\tilde{Z}(u)|\d u+\frac{K_2-K_\sigma}{4\beta^2}\e^{\lambda_0r_0}\frac{1}{r_0}\int_{(-r_0+s)\vee0}^s\e^{\lambda_0 v}f(|\tilde{Z}(v)|)\d v\\
&+\frac{K_2-K_\sigma}{4\beta^2}\e^{\lambda_0 s}f(|\tilde{Z}(s)|),\ \ s\geq 0.
\end{align*}
Combining this with \eqref{cop}, \eqref{fzndv} and \eqref{fzn12}, we arrive at
\begin{align}\label{KGY}
				\e^{\lambda_0 s}\|\tilde{Z}_s\|_{\Gamma^{r_0}}
				\nonumber&\leq \frac{K_2-K_\sigma}{4\beta^2}\e^{\lambda_0r_0}f'(0)\frac{1}{r_0}\int_{-r_0}^0|\tilde{Z}(u)|\d u+\frac{K_2-K_\sigma}{4\beta^2}\e^{\lambda_0r_0}f'(0)|\tilde{Z}(0)|\\
				\nonumber&+\frac{K_2-K_\sigma}{4\beta^2}\e^{\lambda_0r_0}f'(0)K_b\int_0^{s}\e^{\lambda_0 t}\|\tilde{Z}_t\|_{\Gamma^{r_0}}\d t\\
				\nonumber&+\frac{K_2-K_\sigma}{4\beta^2}\e^{\lambda_0r_0}f'(0)K_b\int_0^{s}\e^{\lambda_0 t}\W_1^{\Gamma^{r_0}}(\mu_t,\nu_t)\d t+\frac{K_2-K_\sigma}{4\beta^2}\e^{\lambda_0r_0}\int_0^{s}\e^{\lambda_0 t}\ell(\varepsilon) \d t\\
				\nonumber&+\frac{K_2-K_\sigma}{4\beta^2}\e^{\lambda_0r_0}\frac{1}{r_0}\int_{(-r_0+s)\vee0}^sM_v\d v+\frac{K_2-K_\sigma}{4\beta^2}M_s\\				\nonumber&+\frac{K_2-K_\sigma}{4\beta^2}f'(0)|\tilde{Z}(0)|+\frac{K_2-K_\sigma}{4\beta^2}\int_0^{s}\e^{\lambda_0 t}f'(0)K_b\|\tilde{Z}_t\|_{\Gamma^{r_0}}\d t\\
				&+\frac{K_2-K_\sigma}{4\beta^2}\int_0^{s}\e^{\lambda_0 t}f'(0)K_b\W_1^{\Gamma^{r_0}}(\mu_t,\nu_t)\d t+\frac{K_2-K_\sigma}{4\beta^2}\int_0^{s}\e^{\lambda_0 t}\ell( \varepsilon )\d t\\
				\nonumber&\leq \frac{K_2-K_\sigma}{\beta^2}\e^{\lambda_0r_0}f'(0)\|\tilde{Z}_0\|_{\Gamma^{r_0}}\\
				\nonumber&+\frac{K_2-K_\sigma}{2\beta^2}\e^{\lambda_0r_0}f'(0)K_b\int_0^{s}\e^{\lambda_0 t}\|\tilde{Z}_t\|_{\Gamma^{r_0}}\d t\\
				\nonumber&+\frac{K_2-K_\sigma}{2\beta^2}\e^{\lambda_0r_0}f'(0)K_b\int_0^{s}\e^{\lambda_0 t}\W_1^{\Gamma^{r_0}}(\mu_t,\nu_t)\d t+\frac{K_2-K_\sigma}{2\beta^2}\e^{\lambda_0r_0}\int_0^{s}\e^{\lambda_0 t}\ell( \varepsilon) \d t\\
				\nonumber&+\frac{K_2-K_\sigma}{4\beta^2}\e^{\lambda_0r_0}\frac{1}{r_0}\int_{(-r_0+s)\vee0}^sM_v\d v+\frac{K_2-K_\sigma}{4\beta^2}M_s,\ \ s\geq 0.
			\end{align}
Gronwall's inequality implies that
\begin{align*}
\E\e^{\lambda_0 t}\|\tilde{Z}_t\|_{\Gamma^{r_0}}&\leq \frac{K_2-K_\sigma}{\beta^2}\e^{\lambda_0r_0}f'(0) \e^{\frac{K_2-K_\sigma}{\beta^2}\e^{\lambda_0r_0}f'(0)K_b t}\E\|\tilde{Z}_0\|_{\Gamma^{r_0}}\\
&+\e^{\frac{K_2-K_\sigma}{\beta^2}\e^{\lambda_0r_0}f'(0)K_b t}\frac{K_2-K_\sigma}{2\beta^2}\e^{\lambda_0r_0}\int_0^{t}\e^{\lambda_0 v}\ell( \varepsilon) \d v,\ \ t\geq 0.
\end{align*}
Letting $\varepsilon\to0$, we derive from \eqref{lva} that
\begin{align*}
&\E\|\tilde{Z}_t\|_{\Gamma^{r_0}}\leq \frac{K_2-K_\sigma}{\beta^2}\e^{\lambda_0r_0}f'(0) \e^{\left\{-\lambda_0+\frac{K_2-K_\sigma}{\beta^2}\e^{\lambda_0r_0}f'(0)K_b\right\} t}\E\|\tilde{Z}_0\|_{\Gamma^{r_0}}.
\end{align*}
Taking infimum with respect to $\tilde{X}_0,Y_0$ with $\L_{\tilde{X}_0}=\mu_0, \L_{Y_0}=\nu_0$ and recalling $\delta=f'(0)=\int_0^\infty s\e^{\frac{1}{2\beta^2}\int_0^s\tilde{\gamma}(v)\d v}\d s$ and $\lambda_0=\frac{2\beta^2}{f'(0)}$, we get
\begin{align*}
\W_1^{\Gamma^{r_0}}(P_t^\ast\mu_0,P_t^\ast\nu_0)\leq  c\e^{-\lambda t}\W_1^{\Gamma^{r_0}}(\mu_0,\nu_0)
\end{align*}
for $c=\frac{K_2-K_\sigma}{\beta^2}\e^{\lambda_0r_0}\delta$ and $\lambda=\frac{K_2-K_\sigma}{\beta^2} \e^{\frac{2\beta^2}{\delta}r_0}\delta\left(\frac{2\beta^4}{\e^{\frac{2\beta^2}{\delta}r_0}(K_2-K_\sigma)\delta^2} -K_b\right)>0$ due to \eqref{kb-kd1}.
So, the proof is completed.
\end{proof}
\subsection{Uniform in time propagation of chaos for path dependent mean field interacting particle system}
Let
$N\ge1$ be an integer and $(X_0^i,W^{1,i}(t),W^{2,i}(t))_{1\le i\le N}$ be i.i.d.\,copies of $(X_0,W^1(t),W^2(t))$. For any $1\leq i\leq N$, consider
\begin{equation}\label{iip}
\d X^i(t)=b^{0}(X^i(t))\d t+\int_{\C}\tilde{b}(X_t^i,\xi)\L_{X_t^i}(\d \xi)\d t+\beta\d
W^{1,i}(t)+\sigma(X^i(t))\d W^{2,i}(t),
\end{equation}
here  $\tilde{b}:\C_\infty\times\C_\infty\to\R^d$ satisfying that there exists $K_b\geq0$ such that
\begin{align}\label{lia}&|\tilde{b}(\xi,\eta)-\tilde{b}(\tilde{\xi},\tilde{\eta})|\leq K_b(\|\xi-\tilde{\xi}\|_{\Gamma^{r_0}}+\|\eta-\tilde{\eta}\|_{\Gamma^{r_0}}),\ \ \xi,\tilde{\xi},\eta,\tilde{\eta}\in\mathscr C.
\end{align}
Consider the mean field interacting particle
system:
\begin{equation*}
\d X^{i,N}(t)=b^{0}(X^{i,N}(t))\d t+\frac{1}{N}\sum_{j=1}^N\tilde{b}(X_t^{i,N},X_t^{j,N})\d t+\beta\d
W^{1,i}(t)+\sigma(X^{i,N}(t))\d W^{2,i}(t),
\end{equation*}
where for any $1\leq i\leq N$, $X_0^{i,N}$ is $\C_\infty$-valued and $\F_0$-measurable random variable and the distribution of $(X_0^{i,N})_{1\leq i\leq N}$ is exchangeable.

For any exchangeable $\mu^N\in\scr P((\C_\infty)^N)$, $1\leq m\leq N$, let $(P_t^{m})^\ast\mu^N$
be the distribution of $(X_t^{i,N})_{1\leq i\leq m}$ with initial distribution $\mu^N$. Moreover, let $\mu^{\otimes m}$ denote the $m$ independent product of $\mu$, i.e. $\mu^{\otimes m}=\prod_{i=1}^m\mu$. For any $\xi=(\xi_1,\xi_2,\cdots,\xi_m), \tilde{\xi}=(\tilde{\xi}_1,\tilde{\xi}_2,\cdots,\tilde{\xi}_m)\in\scr \C^m$, define
$$\|\xi-\tilde{\xi}\|_{\tilde{\Gamma}^{r_0}}=\sum_{i=1}^m\|\xi_i-\tilde{\xi}_i\|_{\Gamma^{r_0}}.$$
Define the Wasserstein distance on $\scr P_1((\C_\infty)^m)$ induced by $\tilde{\Gamma}^{r_0}$ as
$$\W_1^{\tilde{\Gamma}^{r_0}}(\mu,\nu):= \inf_{\pi\in\textbf{C}(\mu,\nu)}\int_{\mathscr{C}^m\times \C^m}\|\xi-\tilde{\xi}\|_{\tilde{\Gamma}^{r_0}}\pi(\d \xi,\d \tilde{\xi}),\ \ \mu,\nu\in\scr P_1((\C_\infty)^m),$$
where $\textbf{C}(\mu,\nu)$ is the set of all couplings for $\mu$ and $\nu$.
\begin{thm}\label{POC} Assume {\bf (C)} with \eqref{lia} in place of \eqref{lip1a}. Let $\mu_0^N\in\scr P_{1}((\C_\infty)^N)$ be exchangeable,  $\mu_0\in \scr P_{2}(\C_\infty)$. Assume that
\begin{align}\label{kb-kd}
K_b<\min\left\{\frac{2\beta^4}{\e^{\frac{2\beta^2}{\delta}r_0}(K_2-K_\sigma)\delta^2},
\frac{2(K_2-K_\sigma)\e^{-2(K_2-K_\sigma)r_0}}{6}\right\}
 \end{align}
 for $\delta=\int_0^\infty s\e^{\frac{1}{2\beta^2}\int_0^s\{\gamma(v)+K_\sigma v\}\d v}\d s$.
 Then
\begin{align}\label{CMYa}\nonumber&\W_1^{\tilde{\Gamma}^{r_0}}((P_t^{m})^\ast\mu^N_0,(P_t^\ast\mu_0)^{\otimes m})\\
&\leq c\e^{-\lambda t}\frac{m}{N}\W_1^{\tilde{\Gamma}^{r_0}}(\mu_0^N,\mu_0^{\otimes N})+cm\{1+\{\mu_0(\|\cdot\|_{\Gamma^{r_0}}^{2})\}^{\frac{1}{2}}\}N^{-\frac{1}{2}},\ \ t\geq 0, 1\leq m\leq N
\end{align}
holds for some constants $c,\lambda>0$.
%where
%\begin{equation*}\begin{split}
%R_n(N)=
%\begin{cases}
%N^{-\ff{1}{2}}+N^{-\frac{q-\theta}{q}},~~~~~~~~~~~~~~~~~~~~\theta>\frac{1}{2}, q\neq 2\theta,\\
%N^{-\ff{1}{2}}\log (1+N)+N^{-\frac{q-\theta}{q}},~~~~ ~~\theta=\frac{1}{2}, q\neq 2\theta,\\
%N^{-\ff{\theta}{n}}+N^{-\frac{q-\theta}{q}},~~~~~~~~~~~~~~~~~~~\theta\in[1,\frac{1}{2}), q\neq \frac{n}{n-\theta},
%\end{cases}
%\end{split}\end{equation*}

\end{thm}
\begin{proof}
Simply denote $\mu_t=P_t^\ast\mu_0$.
Let $(\{\tilde{W}^{1,i}(t)\}_{t\geq 0})_{1\leq i\leq N}$ be independent $d$-dimensional Brownian motions independent of $(\{W^{1,i}(t),W^{2,i}(t)\}_{t\geq 0})_{1\leq i\leq N}$. Construct
\begin{equation*}\begin{split}
\d \tilde{X}^{i,N}(t)&=b^{(0)}(\tilde{X}^{i,N}(t))\d t+\int_{\C}\tilde{b}(\tilde{X}^{i,N}_t,\xi)\mu_t(\d \xi)\d t+\sigma(\tilde{X}^{i,N}(t))\d W^{2,i}(t)\\
&+\beta\pi_R^\varepsilon(|\tilde{Z}^{i,N}(t)|)\d
W^{1,i}(t)+\beta\pi_S^\varepsilon(|\tilde{Z}^{i,N}(t)|)\d \tilde{W}^{1,i}(t),
\end{split}\end{equation*}
and
\begin{equation*}\begin{split}
\d Y^{i,N}(t)&=b^{(0)}(Y^{i,N}(t))\d t+\frac{1}{N}\sum_{j=1}^N\tilde{b}(Y^{i,N}_t,Y^{j,N}_t)\d t+\sigma(Y^{i,N}(t))\d W^{2,i}(t)\\
&+\beta\pi_R^\varepsilon(|\tilde{Z}^{i,N}(t)|)(I_{d\times d}-2\tilde{U}^{i,N}_t\otimes \tilde{U}^{i,N}_t)\d
W^{1,i}(t)+\beta\pi_S^\varepsilon(|\tilde{Z}^{i,N}(t)|)\d \tilde{W}^{1,i}(t),
\end{split}\end{equation*}
where  $\tilde{Z}^{i,N}(t)=\tilde{X}^{i,N}(t)-Y^{i,N}(t)$, $\tilde{U}^{i,N}_t=\frac{\tilde{Z}^{i,N}(t)}{|\tilde{Z}^{i,N}(t)|}1_{\{|\tilde{Z}^{i,N}(t)|\neq 0\}}$ and $\L_{(\tilde{X}^{i,N}_0)_{1\leq i\leq N}}=\mu_0^{\otimes N}$ and $\L_{(Y^{i,N}_0)_{1\leq i\leq N}}=\mu_0^N$.
By the It\^{o}-Tanaka formula, \eqref{bslipsg}, \eqref{pdi} and \eqref{lia}, we have
\begin{align*}
\d |\tilde{Z}^{i,N}(t)|&\leq \gamma(|\tilde{Z}^{i,N}(t)|)\d t+K_\sigma|\tilde{Z}^{i,N}(t)|\d t \\
&+\left|\frac{1}{N}\sum_{j=1}^N\tilde{b}(\tilde{X}^{i,N}_t,\tilde{X}^{j,N}_t)-\int_{\C}\tilde{b}(\tilde{X}^{i,N}_t,\xi)\mu_t(\d \xi)\right|\d t\\
&+K_b\|\tilde{Z}^{i,N}_t\|_{\Gamma^{r_0}}\d t+K_b\frac{1}{N}\sum_{j=1}^N\|\tilde{Z}^{j,N}_t\|_{\Gamma^{r_0}}\d t\\
&+\left\<[\sigma(\tilde{X}^{i,N}(t))-\sigma(Y^{i,N}(t))]\d W^{2,i}(t),\frac{\tilde{Z}^{i,N}(t)}{|\tilde{Z}^{i,N}(t)|}1_{\{|\tilde{Z}^{i,N}(t)|\neq 0\}}\right\>\\
&+ 2\beta\pi_R^\varepsilon(|\tilde{Z}^{i,N}(t)|)\left\<\frac{\tilde{Z}^{i,N}(t)}{|\tilde{Z}^{i,N}(t)|}1_{\{|\tilde{Z}^{i,N}(t)|\neq 0\}},\d W^{1,i}(t)\right\>.
\end{align*}
Recall $$\tilde{\gamma}(v)=\gamma(v)+K_{\sigma}v,\ \ v\geq 0,$$ and
$$f(r)=\int_0^r\e^{-\frac{1}{2\beta^2}\int_0^u\tilde{\gamma}(v)\d v}\int_u^\infty s\e^{\frac{1}{2\beta^2}\int_0^s\tilde{\gamma}(v)\d v}\d s\d u,\ \ r\geq 0.$$
%Therefore,
%$$\frac{1}{f'(0)}\W_f\leq \W_1(\mu,\nu)\leq \frac{K_2-K_b}{2\beta}\W_f. $$
By It\^{o}'s formula and \eqref{fii}, we derive
\begin{align}\label{itf1}
\nonumber\d f(|\tilde{Z}^{i,N}(t)|)&\leq f'(|\tilde{Z}^{i,N}(t)|)K_b\|\tilde{Z}^{i,N}_t\|_{\Gamma^{r_0}}\d t+f'(|\tilde{Z}^{i,N}(t)|)K_b\frac{1}{N}\sum_{j=1}^N\|\tilde{Z}^{j,N}_t\|_{\Gamma^{r_0}}\d t\\
\nonumber&+f'(|\tilde{Z}^{i,N}(t)|)\left|\frac{1}{N}\sum_{j=1}^N\tilde{b}(\tilde{X}^{i,N}_t,\tilde{X}^{j,N}_t)-\int_{\C}\tilde{b}(\tilde{X}^{i,N}_t,\xi)\mu_t(\d \xi)\right|\d t\\
\nonumber&+f'(|\tilde{Z}^{i,N}(t)|)\tilde{\gamma}(|\tilde{Z}^{i,N}(t)|)\d t+2\beta^2 f''(|\tilde{Z}^{i,N}(t)|)\pi_R^\varepsilon(|\tilde{Z}^{i,N}(t)|)^2\d t\\
&+f'(|\tilde{Z}^{i,N}(t)|)\left\<[\sigma(\tilde{X}^{i,N}(t))-\sigma(Y^{i,N}(t))]\d W^{2,i}(t),\frac{\tilde{Z}^{i,N}(t)}{|\tilde{Z}^{i,N}(t)|}1_{\{|\tilde{Z}^{i,N}(t)|\neq 0\}}\right\>\\
\nonumber&+f'(|\tilde{Z}^{i,N}(t)|) 2\beta\pi_R^\varepsilon(|\tilde{Z}^{i,N}(t)|)\left\<\frac{\tilde{Z}^{i,N}(t)}{|\tilde{Z}^{i,N}(t)|}1_{\{|\tilde{Z}^{i,N}(t)|\neq 0\}},\d W^{1,i}(t)\right\>.
\end{align}
Similar to \eqref{MKG}, we conclude that
\begin{align}\label{myt}
&f'(|\tilde{Z}^{i,N}(t)|)\tilde{\gamma}(|\tilde{Z}^{i,N}(t)|)+2\beta^2 f''(|\tilde{Z}^{i,N}(t)|)\pi_R^\varepsilon(|\tilde{Z}^{i,N}(t)|)^2\leq -2\beta^2 |\tilde{Z}^{i,N}(t)|+\ell(\varepsilon)
\end{align}
for
\begin{align*}\ell(\varepsilon):=2\beta^2 \varepsilon +f'(0)\left\{\sup_{s\in[0,\varepsilon]}\gamma^{+}(s)+K_\sigma\varepsilon\right\}.
\end{align*}
			Let $$\Delta_t^{i,N}=\left|\frac{1}{N}\sum_{j=1}^N\tilde{b}(\tilde{X}^{i,N}_t,\tilde{X}^{j,N}_t)-\int_{\C}\tilde{b}(\tilde{X}^{i,N}_t,\xi)\mu_t(\d \xi)\right|.$$
			Recall $\lambda_0=\frac{2\beta^2}{f'(0)}$.
It follows from \eqref{cop},  \eqref{itf1} and \eqref{myt} that
\begin{align*}
\d [\e^{\lambda_0 t}f(|\tilde{Z}^{i,N}(t)|)]
&\leq \e^{\lambda_0 t}f'(0)K_b\|\tilde{Z}^{i,N}_t\|_{\Gamma^{r_0}}\d t+\e^{\lambda_0 t}f'(0)K_b\frac{1}{N}\sum_{j=1}^N\|\tilde{Z}^{j,N}_t\|_{\Gamma^{r_0}}\d t\\
&+\e^{\lambda_0 t}\ell(\varepsilon)\d t+\d M_t^{i,N}+\e^{\lambda_0 t}f'(0)\Delta_t^{i,N}\d t
\end{align*}
for \begin{align*}\d M_t^{i,N}&=\e^{\lambda_0 t}f'(|\tilde{Z}^{i,N}(t)|)\left\<[\sigma(\tilde{X}^{i,N}(t))-\sigma(Y^{i,N}(t))]\d W^{2,i}(t),\frac{\tilde{Z}^{i,N}(t)}{|\tilde{Z}^{i,N}(t)|}1_{\{|\tilde{Z}^{i,N}(t)|\neq 0\}}\right\>\\
\nonumber&+\e^{\lambda_0 t}f'(|\tilde{Z}^{i,N}(t)|) 2\beta\pi_R^\varepsilon(|\tilde{Z}^{i,N}(t)|)\left\<\frac{\tilde{Z}^{i,N}(t)}{|\tilde{Z}^{i,N}(t)|}1_{\{|\tilde{Z}^{i,N}(t)|\neq 0\}},\d W^{1,i}(t)\right\>,\ \ t\geq 0.
\end{align*}
By the same argument to derive \eqref{KGY}, we arrive at
\begin{align*}
&\e^{\lambda_0 s}\|\tilde{Z}^{i,N}_s\|_{\Gamma^{r_0}}\\
&\leq \frac{K_2-K_\sigma}{\beta^2}\e^{\lambda_0r_0}f'(0)\|\tilde{Z}^{i,N}_0\|_{\Gamma^{r_0}}\\
&+\frac{K_2-K_\sigma}{2\beta^2}\e^{\lambda_0r_0}f'(0)K_b\int_0^{s}\e^{\lambda_0 t}\|\tilde{Z}^{i,N}_t\|_{\Gamma^{r_0}}\d t\\
&+\frac{K_2-K_\sigma}{2\beta^2}\e^{\lambda_0r_0}f'(0)K_b\int_0^{s}\e^{\lambda_0 t}\frac{1}{N}\sum_{j=1}^N\|\tilde{Z}^{j,N}_t\|_{\Gamma^{r_0}}\d t\\
\nonumber&+\frac{K_2-K_\sigma}{2\beta^2}\e^{\lambda_0r_0}\int_0^{s}\e^{\lambda_0 t}f'(0)\Delta_t^{i,N}\d t+\frac{K_2-K_\sigma}{2\beta^2}\e^{\lambda_0r_0}\int_0^{s}\e^{\lambda_0 t}\ell( \varepsilon) \d t\\
&+\frac{K_2-K_\sigma}{4\beta^2}\e^{\lambda_0r_0}\frac{1}{r_0}\int_{(-r_0+s)\vee0}^sM_v^{i,N}\d v+\frac{K_2-K_\sigma}{4\beta^2}M_s^{i,N},\ \ s\geq 0.
\end{align*}
Applying Gronwall's inequality, we obtain
\begin{align*}
\E\e^{\lambda_0 t}\sum_{i=1}^N\|\tilde{Z}_t^{i,N}\|_{\Gamma^{r_0}}&\leq \frac{K_2-K_\sigma}{\beta^2}\e^{\lambda_0r_0}f'(0) \e^{\frac{K_2-K_\sigma}{\beta^2}\e^{\lambda_0r_0}f'(0)K_b t}\E\sum_{i=1}^N\|\tilde{Z}_0^{i,N}\|_{\Gamma^{r_0}}\\
&+\sum_{i=1}^N\frac{K_2-K_\sigma}{2\beta^2}\e^{\lambda_0r_0}\int_0^{t}\e^{\lambda_0 v}f'(0)\E\Delta_v^{i,N} \d v\\
&+\int_0^t\frac{K_2-K_\sigma}{\beta^2}\e^{\lambda_0r_0}f'(0)K_b\e^{\frac{K_2-K_\sigma}{\beta^2}\e^{\lambda_0r_0}f'(0)K_b (t-s)}\\
&\qquad\qquad\qquad\times\sum_{i=1}^N\frac{K_2-K_\sigma}{2\beta^2}\e^{\lambda_0r_0}\int_0^{s}\e^{\lambda_0 v}f'(0)\E\Delta_v^{i,N} \d v\\
&+N\e^{\frac{K_2-K_\sigma}{\beta^2}\e^{\lambda_0r_0}f'(0)K_b t}\frac{K_2-K_\sigma}{2\beta^2}\e^{\lambda_0r_0}\int_0^{t}\e^{\lambda_0 v}\ell( \varepsilon) \d v,\ \ t\geq 0.
\end{align*}
			Letting $\varepsilon\to0$ and observing \eqref{lva}, we have
\begin{align*}
\sum_{i=1}^N\E\|\tilde{Z}_t^{i,N}\|_{\Gamma^{r_0}}&\leq \frac{K_2-K_\sigma}{\beta^2}\e^{\lambda_0r_0}f'(0) \e^{\left\{-\lambda_0+\frac{K_2-K_\sigma}{\beta^2}\e^{\lambda_0r_0}f'(0)K_b\right\} t}\E\sum_{i=1}^N\|\tilde{Z}_0^{i,N}\|_{\Gamma^{r_0}}\\
&+\sum_{i=1}^N\frac{K_2-K_\sigma}{2\beta^2}\e^{\lambda_0r_0}\int_0^{t}\e^{-\lambda_0(t-v)}f'(0)\E\Delta_v^{i,N} \d v\\
&+\int_0^t\frac{K_2-K_\sigma}{\beta^2}\e^{\lambda_0r_0}f'(0)K_b\e^{\left\{-\lambda_0+\frac{K_2-K_\sigma}{\beta^2}\e^{\lambda_0r_0}f'(0)K_b \right\}(t-s)}\\
&\qquad\qquad\qquad\times\sum_{i=1}^N\frac{K_2-K_\sigma}{2\beta^2}\e^{\lambda_0r_0}\int_0^{s}\e^{-\lambda_0(s- v)}f'(0)\E\Delta_v^{i,N} \d v.
\end{align*}
By \cite[Lemma 2.1]{H2023+}, \eqref{kb-kd} and Theorem \ref{UEt} below, we can find constants $C_0,C_1>0$ such that
\begin{equation}\label{w2}\begin{split}
&\E\Delta_t^{i,N}\leq C_0(1+(\E\|\tilde{X}^{i,N}_t\|_{\Gamma^{r_0}}^2)^{\frac{1}{2}})N^{-\frac{1}{2}}\le C_1\{1+\{\mu_0(\|\cdot\|_{\Gamma^{r_0}}^{2})\}^{\frac{1}{2}}\}N^{-\frac{1}{2}},\ \ t\geq 0.
\end{split}\end{equation}
Taking infimum with respect to $(\tilde{X}^{i,N}_0)_{1\leq i\leq N}, (Y^{i,N}_0)_{1\leq i\leq N}$ with $\L_{(\tilde{X}^{i,N}_0)_{1\leq i\leq N}}=\mu_0^{\otimes N}$ and $\L_{(Y^{i,N}_0)_{1\leq i\leq N}}=\mu_0^N$ and recalling $\delta=f'(0)=\int_0^\infty s\e^{\frac{1}{2\beta^2}\int_0^s\tilde{\gamma}(v)\d v}\d s$ and $\lambda_0=\frac{2\beta^2}{f'(0)}$, we get from \eqref{w2} that 
\begin{align*}
&\W_1^{\tilde{\Gamma}^{r_0}}((P_t^{N})^\ast\mu^N_0,(P_t^\ast\mu_0)^{\otimes N})\leq c\e^{-\lambda t}\W_1^{\tilde{\Gamma}^{r_0}}(\mu_0^N,\mu_0^{\otimes N})+C_2\{1+\{\mu_0(\|\cdot\|_{\Gamma^{r_0}}^{2})\}^{\frac{1}{2}}\}N^{\frac{1}{2}},\ \ t\geq 0
\end{align*}
for some constat $C_2>0$, $c=\frac{K_2-K_\sigma}{\beta^2}\e^{\lambda_0r_0}\delta$ and $\lambda=\frac{K_2-K_\sigma}{\beta^2} \e^{\frac{2\beta^2}{\delta}r_0}\delta\left(\frac{2\beta^4}{\e^{\frac{2\beta^2}{\delta}r_0}(K_2-K_\sigma)\delta^2} -K_b\right)>0$ due to \eqref{kb-kd}.
Combining this with the fact
\begin{align*}
&\W_1^{\Gamma^{r_0}}((P_t^{m})^\ast\mu^N_0,(P_t^\ast\mu_0)^{\otimes m})\leq \frac{m}{N}\W_1^{\Gamma^{r_0}}((P_t^{N})^\ast\mu^N_0,(P_t^\ast\mu_0)^{\otimes N}),\ \ 1\leq m\leq N,t\geq 0,
\end{align*}
we derive \eqref{CMYa} and the proof is completed.
\end{proof}
\section{Appendix}
\subsection{Well-posedness of path dependent McKean-Vlasov SDEs}
In this section, we consider the following path dependent McKean-Vlasov SDEs:
\beq\label{EQ3}\d X(t)=F(X_t,\L_{X_t})\d t+G(X_t,\L_{X_t})\d W(t),\end{equation}
here $W$ is an $n$-dimensional Brownian motion, $F:\C_\infty\times\scr P(\C_\infty)\to\R^d$, $G: \C_\infty\times\scr P(\C_\infty)\to\R^d\otimes\R^n$ are measurable and bounded on bounded sets. To derive the well-posedness of \eqref{EQ3}, we make the following assumptions.
\begin{enumerate}
	\item[\bf{(H)}] There exists $k\geq1 $ such that $F$ is continuous in $\C_\infty\times \mathscr{P}_k(\C_\infty)$, and there exists a constant $K\geq 0$ such that for any $\xi,\eta\in\C$ and $\mu,\nu\in\mathscr{P}_k(\C_\infty)$,
\begin{equation*}\begin{aligned}
	&\ \ \ \ 2 \<F(\xi,\mu)-F(\eta,\nu), \xi(0)-\eta(0) \>\leq K\W^{\C_\infty}_k(\mu,\nu)^2 + K\|\xi-\eta\|_{\infty}^2,
	\end{aligned}\end{equation*}
\begin{equation*}\begin{aligned}
	&\|G(\xi,\mu)-G(\eta,\nu)\|_{HS}^2\leq K\W^{\C_\infty}_k(\mu,\nu)^2 + K\|\xi-\eta\|_{\infty}^2.
	\end{aligned}\end{equation*}
\end{enumerate}
\begin{thm}\label{wep} Assume {\bf(H)}. Then \eqref{EQ3} is well-posed in $\mathscr{P}_k(\C_\infty)$.
\end{thm}
\begin{proof} Fix $T>0$. Let $X_0$ be an $\C_\infty$-valued and $\F_0$-measurable random variable with $\L_{X_0}\in\scr P_k(\C_\infty)$. For $\mu\in C([0,T]; \scr P_k(\C_\infty))$, $t\in[0,T], \xi\in\C$, let $F^\mu(t,\xi) =F(\xi,\mu_t)$, $G^\mu(t,\xi)=G(\xi,\mu_t)$.  Consider
\beq\label{EW1} \d X^\mu(t)=F^\mu(t,X^\mu_t)\d t +G^\mu(t,X^\mu_t)\d W(t),\ \  t\in [0,T].\end{equation}
Let $\mathbf{0}(s)=(0,0,\cdots,0)\in\R^d, s\in[-r_0,0]$. It follows from {\bf (H)} that
\begin{equation*}\begin{split}
	&2 \<F^\mu(t,\xi)-F^\mu(t,\eta), \xi(0)-\eta(0) \>\leq K\|\xi-\eta\|_{\infty}^2,\\
	&\|G^\mu(t,\xi)-G^\mu(t,\eta)\|_{HS}^2\leq K\|\xi-\eta\|_{\infty}^2,\ \ t\in[0,T], \xi,\eta\in\C,
\end{split}\end{equation*}
and
\begin{align}\label{coe}
\nonumber2\<F^\mu(t,\xi),\xi(0)\>
&\leq 2\<F^\mu(t,\xi)-F(\mathbf{0},\delta_\mathbf{0}),\xi(0)\>+|F(\mathbf{0},\delta_\mathbf{0})||\xi(0)|\\
&\leq K(\|\xi\|_\infty^2+\mu_t(\|\cdot\|_\infty^k)^{\frac{2}{k}})+|F(\mathbf{0},\delta_\mathbf{0})||\xi(0)|,\ \ t\in[0,T], \xi\in\C,
\end{align}
\begin{align}\label{cof}
\nonumber\|G^\mu(t,\xi)\|_{HS}^2
&\leq 2\|G^\mu(t,\xi)-G(\mathbf{0},\delta_\mathbf{0})\|_{HS}^2 +2\|G(\mathbf{0},\delta_\mathbf{0})\|_{HS}^2\\
&\leq 2K(\|\xi\|_\infty^2+\mu_t(\|\cdot\|_\infty^k)^{\frac{2}{k}})+2\|G(\mathbf{0},\delta_\mathbf{0})\|_{HS}^2,\ \ t\in[0,T], \xi\in\C.
\end{align}
By \cite[Theorem 2.3]{28}, \eqref{EW1} is strongly well-posed. We should remark that although \cite[Theorem 2.3]{28} only considers the time-homogeneous case, there is no essential difference for the time-inhomogeneous situation. Let $\Phi_t(\mu)=\L_{X^\mu_t}, t\in[0,T]$ for $(X^\mu(t))_{t\in[-r_0,T]}$ solving  \eqref{EW1} with initial value $X_0$. Moreover, for any $p>0$, there exists a constant $C_{T,p}>0$ such that
\begin{align}\label{MoE}
\E\left(\sup_{t\in[-r_0,T]}|X^\mu(t)|^p\bigg|\F_0\right)\leq C_{T,p}\left(1+\|X_0\|_\infty^p+\left(\int_0^T\mu_t(\|\cdot\|_\infty^k)^{\frac{2}{k}}\d t\right)^\frac{p}{2}\right).
\end{align}
In fact, by Jensen's inequality, it is sufficient to prove \eqref{MoE} for $p\geq 2$. By It\^{o}'s formula and \eqref{coe}-\eqref{cof}, it holds
\begin{align}\label{Itf}
\nonumber\d |X^\mu(t)|^2&=2\<F^\mu(t,X^\mu_t), X^\mu(t)\>\d t+\|G^\mu(t,X^\mu_t)\|_{HS}^2\d t+2\<G^\mu(t,X^\mu_t)\d W(t),X^\mu(t)\>\\
&\leq 3K(\|X^\mu_t\|_\infty^2+\mu_t(\|\cdot\|_\infty^k)^{\frac{2}{k}})\d t+|F(\mathbf{0},\delta_\mathbf{0})|| X^\mu(t)|\d t+2\|G(\mathbf{0},\delta_\mathbf{0})\|_{HS}^2\d t\\
\nonumber&+2\<G^\mu(t,X^\mu_t)\d W(t),X^\mu(t)\>,\ \ t\in [0,T].
\end{align}
For any $R>0$, let $\tau_R=\inf\{t\geq 0: |X^\mu(t)|\geq R\}$. Note that for any $p\geq 2$, it follows from BDG's inequality and \eqref{cof} that
\begin{align}\label{TGY}\nonumber&\E\left(\left|\int_0^{s\wedge \tau_R}2\<G^\mu(t,X^\mu_t)\d W(t),X^\mu(t)\>\right|^{\frac{p}{2}}\Big|\F_0\right)\\
\nonumber&\leq \kappa(p)\E\left(\left(\int_0^{s\wedge \tau_R}4\|G^\mu(t,X^\mu_t)\|_{HS}^2|X^\mu(t)|^2\d t\right)^\frac{p}{4}\Big|\F_0\right)\\
&\leq \frac{1}{2}\E\left(\sup_{t\in[-r_0,s\wedge \tau_R]}|X^\mu(t)|^p\Big|\F_0\right)\\
\nonumber&+C_p\int_0^s\left(1+\E\left(\sup_{v\in[-r_0,t\wedge \tau_R]}|X^\mu(v)|^p\Big|\F_0\right)\right)\d t+C_p\left(\int_0^s\mu_t(\|\cdot\|_\infty^k)^{\frac{2}{k}}\d t\right)^{\frac{p}{2}}, \ \ s\geq 0.
\end{align}
Therefore, it follows from \eqref{Itf},  \eqref{TGY} and Gronwall's inequality that
\begin{align}\label{MoE1}
\E\left(\sup_{t\in[-r_0,T\wedge \tau_R]}|X^\mu(t)|^p\bigg|\F_0\right)\leq C_{T,p}\left(1+\|X_0\|_\infty^p+\left(\int_0^T\mu_t(\|\cdot\|_\infty^k)^{\frac{2}{k}}\d t\right)^\frac{p}{2}\right),
\end{align}
which implies $\lim_{R\to\infty}\P(\tau_R<T|\F_0)=0$. Letting $R\to\infty$ in \eqref{MoE1} and applying Fatou's lemma, we obtain \eqref{MoE} immediately.

Next, for any $\mu,\nu\in C([0,T]; \scr P_k(\C_\infty))$, It\^{o}'s formula implies
\begin{align*}
\d |X^\mu(t)-X^\nu(t)|^2&=2\<F^\mu(t,X^\mu_t)-F^\nu(t,X^\nu_t), X^\mu(t)-X^\nu(t)\>\d t\\
 &+\|G^\mu(t,X^\mu_t)-G^\nu(t,X^\nu_t)\|_{HS}^2\d t+\d M_t\\
 &\leq 2K(\|X^\mu_t-X^\nu_t\|_\infty^2+\W^{\C_\infty}_k(\mu_t,\nu_t)^2)\d t+\d M_t,\ \  t\in [0,T],
\end{align*}
where $$\d M_t=2\<[G^\mu(t,X^\mu_t)-G^\nu(t,X^\nu_t)]\d W(t),X^\mu(t)-X^\nu(t)\>.$$
Note that
$$\<M\>_s\leq \int_0^s|X^\mu(t)-X^\nu(t)|^2K(\|X^\mu_t-X^\nu_t\|_\infty^2+\W^{\C_\infty}_k(\mu_t,\nu_t)^2)\d t.$$
For any $p\geq 2$, it follows from BDG's inequality that
\begin{align*}
&\E\left(\sup_{s\in[0,u]}|M_s|^{\frac{p}{2}}\big|\F_0\right)\\
&\leq c_p\E\left\{\left(\int_0^u|X^\mu(t)-X^\nu(t)|^2K(\|X^\mu_t-X^\nu_t\|_\infty^2+\W^{\C_\infty}_k(\mu_t,\nu_t)^2)\d t\right)^{\frac{p}{4}}\bigg|\F_0\right\}\\
&\leq \frac{1}{2}\E\left(\sup_{t\in[0,u]}|X^\mu(t)-X^\nu(t)|^{p}\big|\F_0\right)+ \tilde{c}_{p,T}\int_0^u\E(\|X^\mu_t-X^\nu_t\|_\infty^{p}|\F_0)\d t\\
&+\tilde{c}_{p,T}\left(\int_0^u\W^{\C_\infty}_k(\mu_t,\nu_t)^{2}\d t\right)^{\frac{p}{2}}.
\end{align*}
As a result, Gronwall's inequality gives
\begin{align}\label{kgt}\nonumber\E(\|X^\mu_s-X^\nu_s\|^{p}_\infty|\F_0)&\leq \E\left(\sup_{t\in[-r_0,s]}|X^\mu(t)-X^\nu(t)|^{p}\big|\F_0\right)\\
&\leq C_{p,T}\left(\int_0^s\W^{\C_\infty}_k(\mu_t,\nu_t)^{2}\d t\right)^{\frac{p}{2}}.
\end{align}
Jensen's inequality yields that \eqref{kgt} holds for any $p>0$. Taking $p=k$, we conclude that
\begin{align*}\W_k^{\C_\infty}(\Phi_s(\mu),\Phi_s(\nu))\leq \left(\E(\|X^\mu_s-X^\nu_s\|^{k}_\infty)\right)^{\frac{1}{k}}\leq C_{k,T}\left(\int_0^s\W^{\C_\infty}_k(\mu_t,\nu_t)^{2}\d t\right)^{\frac{1}{2}},\ \ s\in[0,T].
\end{align*}
For any $\eta>0$, it holds
\begin{align*}\sup_{s\in[0,T]}\e^{-\eta s}\W_k^{\C_\infty}(\Phi_s(\mu),\Phi_s(\nu))&\leq C_{k,T}\sup_{s\in[0,T]}\e^{-\eta s}\W^{\C_\infty}_k(\mu_s,\nu_s)\left(\int_0^s\e^{-2\eta (s-t)}\d t\right)^{\frac{1}{2}}\\
&\leq C_{k,T}\frac{\sqrt{2}}{2}\eta^{-\frac{1}{2}}\sup_{s\in[0,T]}\e^{-\eta s}\W^{\C_\infty}_k(\mu_s,\nu_s).
\end{align*}
Taking $\eta_0=2 C_{K,T}^2$, we arrive at
\begin{align}\label{cot}\sup_{s\in[0,T]}\e^{-\eta_0 s}\W_k^{\C_\infty}(\Phi_s(\mu),\Phi_s(\nu))\leq \frac{1}{2}\sup_{s\in[0,T]}\e^{-\eta_0 s}\W^{\C_\infty}_k(\mu_s,\nu_s).
\end{align}
Set
$$E_{0}:= \big\{\mu\in C([0,T]; \scr P_k(\C_\infty)):\mu_0=\L_{X_0}\big\}$$ and equip it with the complete metric
$$ \rr_{\eta_0}(\mu,\nu):= \sup_{t\in[0,T]} \e^{-\eta_0 t}\W_k^{\C_\infty}(\mu_t,\nu_t),\ \ \mu,\nu\in E_{0}.$$
Then \eqref{cot} means that $\Phi$ is strictly contractive in $(E_{0},\rho_{\eta_0})$. Consequently, the Banach fixed point theorem together with the definition of $\Phi$ implies that there exists a unique $\mu\in E_{0}$ such that
 $$\Phi_t(\mu)=\mu_t=\L_{X^\mu_t}, ~~~t\in[0,T].$$
Hence, the proof is completed.
\end{proof}
\subsection{Well-posedness of mean field interacting particle system}
Let $F, G$ be in Section 4.1 and $(W^i)_{i\geq 1}$ be independent $n$-dimensional Brownian motions. Consider
\beq\label{EQ31}\d X^{i,N}(t)=F\left(X_t^{i,N},\frac{1}{N}\sum_{i=1}^N \delta_{X_t^{i,N}}\right)\d t+G\left(X^{i,N}_t,\frac{1}{N}\sum_{i=1}^N \delta_{X_t^{i,N}}\right)\d W^i(t),\ \ 1\leq i\leq N.\end{equation}
\begin{lem}\label{lem}
 Assume {\bf(H)}. Then \eqref{EQ31} admits a unique strong solution and for any $T>0$, there exists a constant $C(T,N,k)>0$ such that
\begin{equation*}
\begin{split}
\sum_{i=1}^N\E\left(\sup_{t\in[-r_0,T]}|X^{i,N}(t)|^2\big|\F_0\right)\leq C(T,N,k)\sum_{i=1}^N\left(1+\|X_0^{i,N}\|_\infty^2\right).
\end{split}
\end{equation*}
\end{lem}
\begin{proof}
For $\xi:=(\xi_1,\xi_2,\cdots,\xi_N)^*\in\C^N$, set $\tt\mu^N_{\xi}=\frac{1}{N}\sum_{i=1}^N\delta_{\xi_i}$ and
\begin{equation*}
\begin{split}
&\hat F(\xi):=\big(F(\xi_1,\tt\mu^N_{\xi}),\cdots,F(\xi_N,\tt\mu^N_{\xi})\big)^*,\\
& \hat G(\xi):=\mbox{diag}\big(G(\xi_1,\tt\mu^N_{\xi}),\cdots,G(\xi_N,\tt\mu^N_{\xi})\big),\\
&\hat W(t):=\big(W^1(t),\cdots,W^N(t)\big)^*,\ \ t\geq 0.
\end{split}
\end{equation*}
Then it is clear that $(\hat W(t))_{t\geq 0}$ is an $N\times n$-dimensional Brownian
motion and \eqref{EQ31} can be reformulated as
\begin{equation}\label{B1}
\d \hat{X}(t)=\hat{F}(\hat{X}_t)\d t+\hat G(\hat{X}_t)\d \hat W(t),\ \ \hat{X}_0=(X_0^{1,N},X_0^{2,N},\cdots,X_0^{N,N})^\ast.
\end{equation}
Note that
\begin{align}\label{W3}
\nonumber&\W_k^{\C_{\infty}}\left(\ff{1}{N}\sum_{i=1}^N\dd_{\xi_i},\ff{1}{N}\sum_{i=1}^N\dd_{\tilde{\xi}_i}\right)^2\\
&\le
\left(\ff{1}{N}\sum_{i=1}^N\|\xi_i-\tilde{\xi}_i\|_\infty^k\right)^{\frac{2}{k}}\\
\nonumber&\leq \left(N^{-1}1_{\{k\in[1,2]\}}+N^{-\frac{2}{k}}1_{\{k\in(2,\infty)\}}\right)\sum_{i=1}^N\|\xi_i-\tilde{\xi}_i\|_\infty^2,\ \ \xi_i,\tilde{\xi}_i\in\C, 1\leq i\leq N.
\end{align}
It follows from {\bf (H)} and \eqref{W3} that $\hat{F}$ is continuous in $(\C_\infty)^N$ and for any $\xi=(\xi_1,\xi_2,\cdots,\xi_N)^*, \eta=(\eta_1,\eta_2,\cdots,\eta_N)^*\in\C^N,1\leq i\leq N$,
\begin{equation}\label{ccb0}\begin{split}
&2\<F(\xi_i,\tt\mu^N_{\xi})-F(\eta_i,\tt\mu^N_{\eta}), \xi_i(0)-\eta_i(0)\>\\
&\leq K\left(\|\xi_i-\eta_i\|_\infty^2+\left(\ff{1}{N}\sum_{i=1}^N\|\xi_i-\eta_i\|_\infty^k\right)^{\frac{2}{k}}\right)\\
&\leq K\left(\|\xi_i-\eta_i\|_\infty^2+(N^{-1}1_{\{k\in[1,2]\}}+N^{-\frac{2}{k}}1_{\{k\in(2,\infty)\}})\sum_{i=1}^N\|\xi_i-\eta_i\|_\infty^2\right)\\
\end{split}\end{equation}
and similarly,
\begin{align}\label{ccB0}
\nonumber&\|G(\xi_i,\tt\mu^N_{\xi})-G(\eta_i,\tt\mu^N_{\eta})\|_{HS}^2\\
&\leq K\left(\|\xi_i-\eta_i\|_\infty^2+(N^{-1}1_{\{k\in[1,2]\}}+N^{-\frac{2}{k}}1_{\{k\in(2,\infty)\}})\sum_{i=1}^N\|\xi_i-\eta_i\|_\infty^2\right).
\end{align}
So, \eqref{ccb0} and \eqref{ccB0} imply that for any $\xi=(\xi_1,\xi_2,\cdots,\xi_N)^*, \eta=(\eta_1,\eta_2,\cdots,\eta_N)^*\in\C^N$,
\begin{equation}\label{ccb}\begin{split}
&2\<\hat{F}(\xi)-\hat{F}(\eta)), \xi(0)-\eta(0)\>\\
&=2\sum_{i=1}^N\<F(\xi_i,\tt\mu^N_{\xi})-F(\eta_i,\tt\mu^N_{\eta}), \xi_i(0)-\eta_i(0)\>\\
&\leq K\left(1+1_{\{k\in[1,2]\}}+N^{1-\frac{2}{k}}1_{\{k\in(2,\infty)\}}\right)\sum_{i=1}^N\|\xi_i-\eta_i\|_\infty^2,
\end{split}\end{equation}
and
\begin{align}\label{ccB}
\nonumber\|\hat{G}(\xi)-\hat{G}(\eta)\|_{HS}^2&\leq \sum_{i=1}^N\|G(\xi_i,\tt\mu^N_{\xi})-G(\eta_i,\tt\mu^N_{\eta})\|_{HS}^2\\
&\leq  K\left(1+1_{\{k\in[1,2]\}}+N^{1-\frac{2}{k}}1_{\{k\in(2,\infty)\}}\right)\sum_{i=1}^N\|\xi_i-\eta_i\|_\infty^2.
\end{align}
Observe
\begin{align}\label{eti}\sum_{i=1}^N\|\xi_i-\eta_i\|_\infty^2\leq N\|\xi-\eta\|_{\infty}^2,\ \ \xi=(\xi_1,\xi_2,\cdots,\xi_N)^*, \eta=(\eta_1,\eta_2,\cdots,\eta_N)^*\in\C^N.
\end{align}
Therefore, according to \cite[Theorem 2.3]{28} and \eqref{ccb}-\eqref{eti}, \eqref{B1} and consequently \eqref{EQ31} admits a unique strong solution $\{(X^{i,N}(t))_{t\geq -r_0}\}_{1\leq i\leq N}$.
%satisfying
%\begin{equation}\label{etsgs}\E \sup_{t\in[-r,T]}|X^N(t)|\le C(T)(1+N\E\|X_0\|_{\infty}).\end{equation}
Finally, by It\^{o}'s formula and \eqref{ccb0}-\eqref{ccB0}, there exists a constant $C>0$ such that
\begin{equation*}
\begin{split}
|X^{i,N}(t)|^2&\leq |X^{i,N}(0)|^2+\int_0^t 2\left\<X^{i,N}(s),G\left(X^{i,N}_s,\frac{1}{N}\sum_{i=1}^N\delta_{X^{i,N}_s}\right) \d W^{i}(s)\right\>\\
&+C\int_0^t\left[1+\|X^{i,N}_s\|^2_\infty+ (N^{-1}1_{\{k\in[1,2]\}}+N^{-\frac{2}{k}}1_{\{k\in(2,\infty)\}})\sum_{j=1}^N\|X^{j,N}_s\|_\infty^2\right]\d s.
\end{split}
\end{equation*}
It follows from BDG's inequality and \eqref{ccB0} that
\begin{align*}
&\E\left(\sup_{t\in[0,u]}\left|\int_0^t 2\left\<X^{i,N}(s),G\left(X^{i,N}_s,\frac{1}{N}\sum_{i=1}^N\delta_{X^{i,N}_s}\right) \d W^{i}(s)\right\>\right|\Big|\F_0\right)\\
&\leq \frac{1}{2}\E\left(\sup_{t\in[-r_0,u]}|X^{i,N}(t)|^2\big|\F_0\right)\\
&+C_0\int_0^u\E\left(\left[1+\|X^{i,N}_s\|^2_\infty+ (N^{-1}1_{\{k\in[1,2]\}}+N^{-\frac{2}{k}}1_{\{k\in(2,\infty)\}})\sum_{j=1}^N\|X^{j,N}_s\|_\infty^2\right]\Big|\F_0\right)\d s.
\end{align*}
%\begin{equation*}
%\begin{split}
%V_\vv(X^{i,N}(t))&\leq |X^{i}(0)|^2+C\int_0^t\left[1+|X^{i,N}(s)|+\ff{1}{N}\sum_{j=1}^N|X^{j,N}(s)|\right]\d s\\
%&+C\int_0^t\|X^{i,N}_s\|_{L^1(m)}\d s+\int_0^t V'_\vv(X^{i,N}(s))\si(s,X^{i,N}(s)) \d W^{i}(s)\\
%&+\frac{1}{2}\int_0^t V''_\vv(X^{i,N}(s))\si(s,X^{i,N}(s))^2\d s.
%\end{split}
%\end{equation*}
Gronwall's inequality implies
\begin{equation*}
\begin{split}
\sum_{i=1}^N\E\left(\sup_{t\in[-r_0,T]}|X^{i,N}(t)|^2\big|\F_0\right)\leq C(T,N,k)\sum_{i=1}^N\left(1+\|X_0^{i,N}\|_\infty^2\right)
\end{split}
\end{equation*}
for some constant $C(T,N,k)>0$.
So, we complete the proof.
\end{proof}
\subsection{Uniform in time estimate for the second moment of McKean-Vlasov SDEs}
Finally, we give a  uniform in time estimate for the second moment of $X_t^1$ in \eqref{iip}. For simplicity, we denote $X_t=X_t^1$ and let $\mu_t$ be the distribution of $X_t$.

\begin{thm}\label{UEt} Assume {\bf(C)} with \eqref{lia} in place of \eqref{lip1a} and
\begin{align}\label{CTK0}
K_b<\frac{2(K_2-K_\sigma)\e^{-2(K_2-K_\sigma)r_0}}{6}.
\end{align}
Then there exists a constant $C>0$ such that
\begin{align}\label{use}\sup_{t\geq 0}(P_t^\ast\mu_0)(\|\cdot\|^2_{\Gamma^{r_0}})\leq C(1+\mu_0(\|\cdot\|^2_{\Gamma^{r_0}})).
\end{align}
\end{thm}
\begin{proof}
Recall $\mathbf{0}(s)=(0,0,\cdots,0)\in\R^d, s\in[-r_0,0]$.
By {\bf(C)} with \eqref{lia} in place of \eqref{lip1a}, we can find a constant $C_0>0$ such that
\begin{align}\label{mon0}
\nonumber&2\<x,b^{(0)}(x)\>+d\beta^2 +\|\sigma(x)\|_{HS}^2+2|x||\tilde{b}(\mathbf{0},\mathbf{0})|\\
\nonumber&\leq (2K_1+2K_2)|x|^21_{\{|x|\leq 2R\}}-(2K_2-2K_\sigma)|x|^2+2\<x,b^{(0)}(0)\>+d\beta^2\\
\nonumber&+2\sqrt{2K_\sigma}\|\sigma(0)\|_{HS}|x|+\|\sigma(0)\|_{HS}^2+2|x||\tilde{b}(\mathbf{0},\mathbf{0})|\\
&\leq (2K_1+2K_2)4R^2+d\beta^2+\|\sigma(0)\|_{HS}^2-(2K_2-2K_\sigma)|x|^2\\
\nonumber&+2|x|(|b^{(0)}(0)|+\sqrt{2K_\sigma}\|\sigma(0)\|_{HS}+|\tilde{b}(\mathbf{0},\mathbf{0})|)\\
\nonumber&\leq C_0-(2K_2-2K_\sigma)|x|^2+\frac{2(K_2-K_\sigma)\e^{-2(K_2-K_\sigma)r_0}-6K_b}{4}|x|^2,\ \ x\in\R^d,
\end{align}
and
\begin{align}\label{mon1}
\nonumber&2\left\<\xi(0),\int_{\C}\tilde{b}(\xi,\eta)\mu_t(\d \eta)-\tilde{b}(\mathbf{0},\mathbf{0})\right\>\\
\nonumber&\leq 2|\xi(0)|K_b(\|\xi\|_{\Gamma^{r_0}}+\mu_t(\|\cdot\|_{\Gamma^{r_0}}))\\
&\leq K_b\|\xi\|_{\Gamma^{r_0}}^2+K_b|\xi(0)|^2+K_b\mu_t(\|\cdot\|_{\Gamma^{r_0}})^2+K_b|\xi(0)|^2\\
\nonumber&\leq 2K_b|\xi(0)|^2+K_b\|\xi\|_{\Gamma^{r_0}}^2+K_b\mu_t(\|\cdot\|_{\Gamma^{r_0}})^2.
\end{align}
Let \begin{align}\label{la1}\lambda_1=2(K_2-K_\sigma), \ \ \varepsilon_0=\frac{\lambda_1\e^{-\lambda_1r_0}-6K_b}{4}.
\end{align}
Combining \eqref{mon0}-\eqref{la1} with It\^{o}'s formula, we derive
\begin{align}\label{fzn00}
\nonumber&\d [\e^{\lambda_1 t}|X(t)|^2]\\
&\leq \e^{\lambda_1 t}C_0\d t+\e^{\lambda_1 t}(\varepsilon_0+2K_b)|X(t)|^2\d t+\e^{\lambda_1 t}K_b(\|X_t\|_{\Gamma^{r_0}}^2+\mu_t(\|\cdot\|_{\Gamma^{r_0}}^2))\d t+\e^{\lambda_1 t}\d \bar{M}_t
\end{align}
for some martingale $\bar{M}_t$  under $\P(\cdot|\F_0)$. \eqref{fzn00} gives
\begin{align}\label{fzn01}
\nonumber \e^{\lambda_1 v}|X(v)|^2&\leq |X(0)|^2+\int_0^{v}\e^{\lambda_1 t}C_0\d t+\int_0^v\e^{\lambda_1 t}(\varepsilon_0+2K_b)|X(t)|^2\d t\\
&+\int_0^{v}\e^{\lambda_1 t}K_b(\|X_t\|_{\Gamma^{r_0}}^2+\mu_t(\|\cdot\|_{\Gamma^{r_0}}^2))\d t+\int_0^v\e^{\lambda_1 t}\d \bar{M}_t,\ \ v\geq 0.
\end{align}
So, we get
\begin{align}\label{fzn05}
\nonumber&\frac{1}{r_0}\int_{(-r_0+t)\vee0}^t\e^{\lambda_1 v}|X(v)|^2\d v\\
\nonumber&\leq \frac{1}{r_0}\int_{(-r_0+t)\vee0}^t|X(0)|^2\d v+\frac{1}{r_0}\int_{(-r_0+t)\vee0}^t\int_0^{v}\e^{\lambda_1 s}C_0\d s\d v\\
\nonumber&+(\varepsilon_0+2K_b)\frac{1}{r_0}\int_{(-r_0+t)\vee0}^t\int_0^{v}\e^{\lambda_1 s}|X(s)|^2\d s\d v\\
&+\frac{1}{r_0}\int_{(-r_0+t)\vee0}^t\int_0^{v}\e^{\lambda_1 s}K_b(\|X_s\|_{\Gamma^{r_0}}^2+\mu_s(\|\cdot\|_{\Gamma^{r_0}}^2))\d s\d v\\
&\nonumber+\frac{1}{r_0}\int_{(-r_0+t)\vee0}^t\int_0^v\e^{\lambda_1 s}\d \bar{M}_s\d v\\
\nonumber&\leq |X(0)|^2+\int_0^{t}\e^{\lambda_1 s}C_0\d s+(\varepsilon_0+2K_b)\int_0^{t}\e^{\lambda_1 s}|X(s)|^2\d s\\
\nonumber&+\int_0^{t}\e^{\lambda_1 s}K_b(\|X_s\|_{\Gamma^{r_0}}^2+\mu_s(\|\cdot\|_{\Gamma^{r_0}}^2))\d s+\frac{1}{r_0}\int_{(-r_0+t)\vee0}^t\int_0^v\e^{\lambda_1 s}\d \bar{M}_s\d v.
\end{align}
By the definition of $\|\cdot\|_{\Gamma^{r_0}}$, we arrive at
\begin{align*}
\e^{\lambda_1 t}\|X_t\|_{\Gamma^{r_0}}^2&\leq\frac{1}{2}\e^{\lambda_1 t}\frac{1}{r_0}\int_{-r_0}^0|X(t+u)|^2\d u+\frac{1}{2}\e^{\lambda_1 t}|X(t)|^2\\
&=\frac{1}{2}\e^{\lambda_1 r_0}\frac{1}{r_0}\int_{-r_0}^0\e^{\lambda_1 (t-r_0)}|X(t+u)|^2\d u+\frac{1}{2}\e^{\lambda_1 t}|X(t)|^2\\
&\leq\frac{1}{2}\e^{\lambda_1 r_0}\frac{1}{r_0}\int_{-r_0}^0\e^{\lambda_1 (t+u)^+}|X(t+u)|^2\d u+\frac{1}{2}\e^{\lambda_1 t}|X(t)|^2\\
&\leq  \frac{1}{2}\e^{\lambda_1r_0}\frac{1}{r_0}\int_{-r_0}^0|X(u)|^2\d u+\frac{1}{2}\e^{\lambda_1r_0}\frac{1}{r_0}\int_{(-r_0+t)\vee0}^t\e^{\lambda_1 v}|X(v)|^2\d v+\frac{1}{2}\e^{\lambda_1 t}|X(t)|^2.
\end{align*}
Combining this with \eqref{fzn01}-\eqref{fzn05}, we derive
\begin{align*}
\e^{\lambda_1 t}\|X_t\|_{\Gamma^{r_0}}^2&\leq \frac{1}{2}\e^{\lambda_1r_0}\frac{1}{r_0}\int_{-r_0}^0|X(u)|^2\d u\\
&+\frac{1}{2}\e^{\lambda_1r_0}|X(0)|^2+\frac{1}{2}\e^{\lambda_1r_0}\int_0^{t}\e^{\lambda_1 s}C_0\d s+\frac{1}{2}(\varepsilon_0+2K_b)\e^{\lambda_1r_0}\int_0^{t}\e^{\lambda_1 s}|X(s)|^2\d s\\
\nonumber&+\frac{1}{2}\e^{\lambda_1r_0}\int_0^{t}\e^{\lambda_1 s}K_b(\|X_s\|_{\Gamma^{r_0}}^2+\mu_s(\|\cdot\|_{\Gamma^{r_0}}^2))\d s+\frac{1}{2}\e^{\lambda_1r_0}\frac{1}{r_0}\int_{(-r_0+t)\vee0}^t\int_0^v\e^{\lambda_1 s}\d \bar{M}_s\d v\\
&+\frac{1}{2} |X(0)|^2+\frac{1}{2}\int_0^{t}\e^{\lambda_1 s}C_0\d s+\frac{1}{2}\int_0^t\e^{\lambda_1 s}(\varepsilon_0+2K_b)|X(s)|^2\d s\\
&+\frac{1}{2}\int_0^{t}\e^{\lambda_1 s}K_b(\|X_s\|_{\Gamma^{r_0}}^2+\mu_s(\|\cdot\|_{\Gamma^{r_0}}^2))\d s+\frac{1}{2}\int_0^t\e^{\lambda_1 s}\d \bar{M}_s.
\end{align*}
This together with the fact $|X(s)|^2\leq 2\|X_s\|_{\Gamma^{r_0}}^2$ implies that
\begin{align*}
\e^{\lambda_1 t}\|X_t\|_{\Gamma^{r_0}}^2
&\leq 2\e^{\lambda_1r_0}\|X_0\|_{\Gamma^{r_0}}^2+\e^{\lambda_1r_0}\int_0^{t}\e^{\lambda_1 s}C_0\d s+(\varepsilon_0+2K_b)\e^{\lambda_1r_0}\int_0^{t}\e^{\lambda_1 s}|X(s)|^2\d s\\
\nonumber&+\e^{\lambda_1r_0}\int_0^{t}\e^{\lambda_1 s}K_b(\|X_s\|_{\Gamma^{r_0}}^2+\mu_s(\|\cdot\|_{\Gamma^{r_0}}^2))\d s\\
&+\frac{1}{2}\e^{\lambda_1r_0}\frac{1}{r_0}\int_{(-r_0+t)\vee0}^t\int_0^v\e^{\lambda_1 s}\d \bar{M}_s\d v+\frac{1}{2}\int_0^t\e^{\lambda_1 s}\d \bar{M}_s\\
&\leq 2\e^{\lambda_1r_0}\|X_0\|_{\Gamma^{r_0}}^2+\e^{\lambda_1r_0}\int_0^{t}\e^{\lambda_1 s}C_0\d s\\
\nonumber&+(2\varepsilon_0+5K_b)\e^{\lambda_1r_0}\int_0^{t}\e^{\lambda_1 s}\|X_s\|_{\Gamma^{r_0}}^2\d s
+\e^{\lambda_1r_0}\int_0^{t}\e^{\lambda_1 s}K_b\mu_s(\|\cdot\|_{\Gamma^{r_0}}^2)\d s\\
&+\frac{1}{2}\e^{\lambda_1r_0}\frac{1}{r_0}\int_{(-r_0+t)\vee0}^t\int_0^v\e^{\lambda_1 s}\d \bar{M}_s\d v+\frac{1}{2}\int_0^t\e^{\lambda_1 s}\d \bar{M}_s.
\end{align*}
Applying Gronwall's inequality, we conclude
\begin{align*}
\E\e^{\lambda_1 t}\|X_t\|_{\Gamma^{r_0}}^2&\leq 2\e^{\lambda_1r_0} \e^{\e^{\lambda_1r_0}(6K_b +2\varepsilon_0) t}\E\|X_0\|_{\Gamma^{r_0}}^2+\e^{\lambda_1r_0}\int_0^{t}\e^{\lambda_1 s}C_0\d s\\
&+\int_0^t\e^{\lambda_1r_0}(6K_b+2\varepsilon_0)\e^{\e^{\lambda_1r_0}(6K_b+2\varepsilon_0)(t-s)}\int_0^s\e^{\lambda_1 r_0}\e^{\lambda_1 v}C_0\d v\d s.
\end{align*}
This combined with \eqref{la1} yields
\begin{align*}
\E\|X_t\|_{\Gamma^{r_0}}^2&\leq 2\e^{\lambda_1r_0} \e^{-\frac{1}{2}(\lambda _1-\e^{\lambda_1r_0}6K_b ) t}\E\|X_0\|_{\Gamma^{r_0}}^2+\e^{\lambda_1r_0}\int_0^{t}\e^{-\lambda_1 (t-s)}C_0\d s\\
&+\int_0^t\e^{\lambda_1r_0}(6K_b+2\varepsilon_0)\e^{-\frac{1}{2}(\lambda _1-\e^{\lambda_1r_0}6K_b ) (t-s)}\int_0^s\e^{\lambda_1 r_0}\e^{-\lambda_1(s- v)}C_0\d v\d s,
\end{align*}
which together with \eqref{la1} and $\lambda_1>\e^{\lambda_1r_0}6K_b$ due to \eqref{CTK0} completes the proof.
\end{proof}

\beg{thebibliography}{99}

\bibitem{MAW} M. Arnaudon, A. Thalmaier, F.-Y. Wang, Gradient estimates and Harnack Inequalities on non-compact Riemannian manifolds,  \emph{Stochastic Process. Appl.}  119(2009), 3653-360.

\bibitem{BPR} R. Baldasso, A. Pereira, G. Reis, Large deviations for interacting diffusions with path-dependent McKean-Vlasov limit, \emph{Ann. Appl. Probab.}  32(2022), 665-695.

\bibitem{BRW20} J. Bao, P. Ren, F.-Y. Wang, Bismut formulas for Lions derivative of McKean-Vlasov SDEs with memory, \emph{J. Differential Equations}  282(2021), 285-329.

\bibitem{DEGZ} A. Durmus,  A. Eberle, A. Guillin, R. Zimmer, An elementary approach to uniform in time propagation of chaos, \emph{Proc. Amer. Math. Soc.} 148(2020), 5387-5398.

\bibitem{EGZ}  A. Eberle, A. Guillin, R. Zimmer, Quantitative Harris-type theorems for diffusions and McKean-
Vlasov processes, \emph{Trans. Amer. Math. Soc.} 371(2019), 7135-7173.

\bibitem{GS} X. Gu, Y. Song, Large and moderate deviation principles for path-distribution-dependent stochastic differential equations, \emph{Discrete Contin. Dyn. Syst. Ser. S}  16(2023), 901-923.

\bibitem{GBM} A. Guillin, P. Le Bris, P. Monmarch\'{e}, Uniform in time propagation of chaos for the 2D vortex model and other singular stochastic systems, \emph{J. Eur. Math. Soc.} (2024), DOI 10.4171/JEMS/1413.

\bibitem{Han} Y. Han, Solving McKean-Vlasov SDEs via relative entropy, \emph{arXiv:2204.05709v3}.

\bibitem{HX} X. Huang,  Path-distribution dependent SDEs with singular coefficients,  \emph{Electron. J. Probab.} 26(2021), 1-21.

\bibitem{HL} X. Huang, W. Lv,  Exponential Ergodicity and propagation of chaos for path-distribution dependent stochastic Hamiltonian system, \emph{Electron. J. Probab.} 28(2023), Paper No. 134, 20 pp.
\bibitem{H2023+} X. Huang,    Propagation of Chaos for Mean Field Interacting Particle System with Multiplicative Noise, \emph{arXiv:2308.15181v6}.

\bibitem{HRW} X. Huang, M. R\"{o}ckner, F.-Y. Wang, Non-linear Fokker--Planck equations for probability measures on path space and path-distribution dependent SDEs, \emph{Discrete Contin. Dyn. Syst.} 39(2019), 3017-3035.

\bibitem{HW2209} X. Huang, F.-Y. Wang, Regularities and Exponential Ergodicity in Entropy for SDEs Driven by Distribution Dependent Noise, \emph{Bernoulli} 30(2024), 3303-3323.

\bibitem{53} K. It\^o, M. Nisio, On stationary solutions of a stochastic differential equation, \emph{J. Math. Kyoto Univ.} 4(1964), 41-75.

\bibitem{JW} P.-E. Jabin,  Z. Wang, Quantitative estimates of propagation of chaos for stochastic systems with $W^{-1,\infty}$ kernels, \emph{Invent. Math.} 214(2018), 523-591.

\bibitem{L18} D. Lacker, On a strong form of propagation of chaos for McKean-Vlasov equations,  \emph{Electron. Commun. Probab.} 23(2018), Paper No. 45, 11 pp.

\bibitem{L21} D. Lacker, Hierarchies, entropy, and quantitative propagation of chaos for mean field diffusions, \emph{Probab. Math. Phys.} 4(2023), 377-432.

\bibitem{LL} D. Lacker, L. Le Flem, Sharp uniform-in-time propagation of chaos, \emph{Probab. Theory Related Fields} 187(2023), 443-480.

\bibitem{LMW} M. Liang, M. B. Majka, J. Wang, Exponential ergodicity for SDEs and McKean-Vlasov processes with L\'{e}vy noise, \emph{Ann. Inst. Henri Poincar\'e Probab. Stat.} 57(2021), 1665-1701.

\bibitem{MRW} P. Monmarch\'{e}, Z. Ren, S. Wang, Time-uniform log-Sobolev inequalities and applications to propagation of chaos, \emph{arXiv:2401.07966}.

\bibitem{LWZ} W. Liu, L. Wu, C. Zhang, Long-time behaviors of mean-field interacting particle systems related to McKean-Vlasov equations, \emph{Comm. Math. Phys.} 387(2021), 179-214.

\bibitem{M03} F. Malrieu,  Convergence to equilibrium for granular media equations and their Euler schemes, \emph{Ann. Appl. Probab.} 13(2003), 540-560.

\bibitem{M} H. P. McKean, A class of Markov processes associated with nonlinear parabolic equations, \emph{Proc Natl Acad Sci U S A} 56(1966), 1907-1911.

\bibitem{29} S.-E. A. Mohammed, \emph{Stochastic functional differential equations,} Pitman Advanced Publishing Program, Boston, 1984.

\bibitem{Pin} M. S. Pinsker, \emph{ Information and Information Stability of Random Variables and Processes,} Holden-Day, San Francisco, 1964.

\bibitem{RW}  P. Ren, F.-Y. Wang, Exponential convergence in entropy and Wasserstein for McKean-Vlasov SDEs, \emph{Nonlinear Anal.} 206(2021), 112259.

 \bibitem{23RW} P. Ren, F.-Y. Wang, Entropy estimate between diffusion processes and application to McKean-Vlasov SDEs,  \emph{arXiv:2302.13500}.

\bibitem{RWu19} P. Ren, J.-L. Wu, Least squares estimator for path-dependent McKean-Vlasov SDEs via discrete-time observations, \emph{Acta Math. Sci. Ser. B} 39(2019), 691-716.

\bibitem{28} Max-K. von Renesse, M. Scheutzow, Existence and uniqueness of solutions of stochastic functional differential equations, \emph{Random Oper. Stoch. Equ.} 18(2010), 267-284.

\bibitem{Schuh} K. Schuh, Global contractivity for Langevin dynamics with distribution-dependent forces and uniform in time propagation of chaos, \emph{Ann. Inst. Henri Poincar\'{e} Probab. Stat.} 60(2024), 753-789.

\bibitem{SWY} J. Shao, F.-Y. Wang, C. Yuan, Harnack inequalities for stochastic (functional) differential equations with non-Lipschitzian coefficients,  \emph{Elect. J. Probab.} 17(2012), 1-18.

\bibitem{Song} Y. Song,  Gradient estimates and exponential ergodicity for Mean-Field SDEs with jumps, \emph{J. Theoret. Probab.} 33(2020), 201-238.

\bibitem{W23} F.-Y. Wang, Exponential ergodicity for non-dissipative McKean-Vlasov SDEs, \emph{Bernoulli} 29(2023), 1035-1062.
\bibitem{FYW1} F.-Y. Wang, Distribution-dependent SDEs for Landau type equations, \emph{Stochatic Process. Appl.} 128(2018), 595-621.

\bibitem{HARNACK} F.-Y. Wang, \emph{Harnack Inequality for Stochastic Partial Differential Equations,} Springer, New York, 2013.

\bibitem{W15} F.-Y. Wang, Asymptotic couplings by reflection and applications for nonlinear monotone SPDEs, \emph{Nonlinear Anal.} 117(2015), 169-188.

\bibitem{WRbook} F.-Y. Wang, P. Ren, Distribution dependent stochastic differential equations, World Scientific(2024).

\end{thebibliography}

\end{document}